\documentclass[a4paper,oneside,11pt]{article} % Indique que le document est de type standard ("article"), au format A4 ("a4"), en recto simple ("oneside"), avec des fontes de taille moyenne ("11pt").

\usepackage[bottom]{footmisc}

\usepackage{a4wide}

\usepackage{amssymb}
\usepackage{tikz}
\usepackage{amsmath}

\usepackage{enumitem, hyperref}
\makeatletter
\def\namedlabel#1#2{\begingroup
    #2%
    \def\@currentlabel{#2}%
    \phantomsection\label{#1}\endgroup
}
\makeatother

\usepackage[all,cmtip]{xy}
\usepackage{tikz-cd}

\usepackage[all]{xy}

\usepackage[utf8]{inputenc} % Encodage des caractères du fichier-source.
\usepackage[T1]{fontenc} % Encodage des caractères en sortie.
\usepackage{textcomp} % Jeu de symboles complémentaires.
\usepackage[english]{babel}
\usepackage[autolanguage]{numprint} %% Formatage des nombres.
\usepackage{hyperref} % Génère des liens hypertexte dans le fichier pdf.
\usepackage{graphicx} % Permet l'insertion d'images.
\usepackage{verbatim} % Définit des environnements de texte préformaté et de commentaires.

\usepackage{amsmath,amssymb,amsthm} % Toutes ces extensions proviennent de la classe AMS-LaTeX.
\usepackage{comment}

\usepackage{fancyhdr} % Extension pour créer les en-têtes personnalisés.
\renewcommand\[{\begin{equation}}\renewcommand\]{\end{equation}} % Avec cette ligne, vos formules seront systématiquement numérotées.
\renewcommand\epsilon\varepsilon % Graphie du symbole "epsilon".
\renewcommand\phi\varphi % Graphie du symbole "phi".

\newtheorem{innercustomcor}{Corollary}
\newenvironment{customcor}[1]
  {\renewcommand\theinnercustomcor{#1}\innercustomcor}
  {\endinnercustomcor}

\newtheorem{innercustomclm}{Claim}
\newenvironment{customclm}[1]
  {\renewcommand\theinnercustomclm{#1}\innercustomclm}
  {\endinnercustomclm}

\newtheorem{innercustomthm}{Theorem}
\newenvironment{customthm}[1]
  {\renewcommand\theinnercustomthm{#1}\innercustomthm}
  {\endinnercustomthm}

\newcommand\NN{\mathcal{N}}
\newcommand\ZZ{\mathbb{Z}} % Ensemble des entiers relatifs.
\newcommand\ab\allowbreak % Raccourci pour \allowbreak.

\newcommand\VV{\mathcal{V}}

\newcommand\Spec{\operatorname{Spec}}

\newcommand\colim{\operatorname{colim}}

\newcommand\CHt{\widetilde{\operatorname{CH}}}

\newcommand\Hom{\operatorname{Hom}}
\newcommand\id{\operatorname{id}}

\newcommand\res{\operatorname{res}}
\newcommand\cores{\operatorname{cores}}
\newcommand\Id{\operatorname{Id}}

\newcommand\Det{\operatorname{Det}}
\newcommand\Flag{\operatorname{Flag}}

\newcommand\eeta{\boldsymbol{\eta}}
\newcommand\Om{\operatorname{\Omega}}
\newcommand\Ab{\mathcal{A}b}

\newcommand\KMW{\underline{\operatorname{K}}^{MW}}

\newcommand\TT{\mathcal{T}}

\newcommand\AAA{\mathbb{A}}
\newcommand\rk{\operatorname{rk}}
\newcommand\CatM{\mathfrak{M}^{\operatorname{M}}}
\newcommand\CatMW{\mathfrak{M}^{\operatorname{MW}}}

\newcommand\kMW{\mathbf{K}^{\text{MW}}}

\providecommand{\keywords}[1]
{
  \small	
  \textbf{\textit{Keywords---}} #1
}

\providecommand{\Codes}[1]
{
  \small	
  \textbf{\textit{MSC---}} #1
}

\newcommand\PP{\mathbb{P}}
\newcommand\codim{\operatorname{codim}}

\theoremstyle{definition} % Pour tout ce qui ressemble à des définitions : définitions, notations...
\newtheorem{Def}{Definition}[section] % Définitions. Avec l'option "[section]", les définitions (et tout le reste) seront numérotées en fonction de la section courante.
\theoremstyle{plain} % Pour tous ce qui ressemble à des théorèmes.
\newtheorem{Pro}[Def]{Proposition} % Propositions. Avec l'option "[Def]", les propositions (et tout le reste) seront numérotées collectivement avec les définitions.
\newtheorem{Lem}[Def]{Lemma} % Lemmes.
\newtheorem{Claim}[Def]{Claim} % Théorèmes.
\newtheorem{The}[Def]{Theorem} % Théorèmes.
\newtheorem{Cor}[Def]{Corollary} % Corollaires.
\theoremstyle{remark} % Pour tout ce qui ressemble à des remarques.
\newtheorem{Exe}[Def]{Example} % Exemples.
\newtheorem{Rem}[Def]{Remark} % Remarques.
\newtheorem{Par}[Def]{} % Paragraphe.

\title{Milnor-Witt Cycle Modules} % Remplacez "Titre" par votre titre.
\author{\sc Niels FELD\footnote{Address: Institut Fourier, 100 Rue des Mathématiques, Grenoble, France.}
\footnote{E-mail address: <niels.feld@univ-grenoble-alpes.fr>.}  } % Remplacez "Auteur" par le nom du ou des auteurs.

\date{Q4 2018} % Avec l'argument "\today", la date indiquée sera celle du jour de la compilation.

\begin{document} % Début du document proprement dit.

\maketitle % Crée le bandeau de titre.
%\begin{abstract} % Résumé. Supprimer ces lignes si le texte ne s'y prête pas.
% Après quelques rappels concernant la théorie des catégories et celle des topos, on définit et étudie certaines propriété du topos effectif introduit par Hyland dans [?].
%\end{abstract}

%\tableofcontents % Table des matières. À réserver aux documents vraiment longs (au moins 25 pages).

\begin{abstract}
We generalize Rost's theory of cycle modules \cite{Rost96} using the Milnor-Witt K-theory instead of the classical Milnor K-theory. We obtain a (quadratic) setting to study general cycle complexes and their (co)homology groups. The standard constructions are developed: proper pushfoward, (essentially) smooth pullback, long exact sequences, spectral sequences and products, as well as the homotopy invariance property; in addition, Gysin morphisms for lci morphisms are constructed. We prove an
adjunction theorem linking our theory to Rost's. This work extends Schmid's thesis \cite{Schmid98}.
\end{abstract}


\Codes{14C17, 14C35, 11E81}
\tableofcontents

\section{Introduction}

\subsection*{Historical approach}
Back in the nineties, Rost developed the theory of cycle modules \cite{Rost96} which was used in the famous proof of Milnor conjecture by Voevodsky (see \cite{SusVoe00} or \cite{Voe03}). Indeed, a crucial step in the proof was the construction of the motive associated to some quadratic forms \cite{Rost98}. In \cite{Rost96}, a cycle module $M$ over a perfect field $k$ is the data of a $\ZZ$-graded abelian group $M(E)$ for every finitely generated field extension $E/k$, equipped with restriction maps, (finite) corestriction maps, a Milnor K-theory module action and residue maps $\partial$. These data are subject to some axioms $(r1a),\dots, (r3e)$ such that Milnor K-theory groups form a Rost cycle module. From the formal definitions, Rost produced a general theory that encompassed previous ones (Quillen K-theory and étale cohomology, for instance). Moreover, Rost's theory is fundamentally linked with the theory of motives: cycle modules can be realized geometrically. This can be illustrated by the following theorem.
\begin{customthm}{1}[Déglise]\cite{Deg03}\label{theseDeglise}
Let $k$ be a perfect field. The category of Rost cycle modules over $k$ is equivalent to the heart of the category of Voevodsky’s motives $\operatorname{DM}(k,\ZZ)$ with respect to the homotopy t-structure.
\end{customthm}
The idea of $\AAA^1$-homotopy theory, due to Morel and Voevodsky, was to apply techniques from algebraic topology to the study of schemes (the affine line $\AAA^1$ playing the role of the unit interval $[0,1]$). This idea gave rise to many important results and new categories, such as $\operatorname{DM}(k,\ZZ)$ mentioned above and $\operatorname{SH}(k)$, the stable homotopy category \cite{Morel99}. 
\par In his study of $\operatorname{SH}(k)$, Morel (in joint work with Hopkins) defined for a field $E$ the Milnor-Witt K-theory $\kMW_*(E)$ (see \cite[Definition 3.1]{Mor12}). This $\ZZ$-graded abelian group behaves in positive degrees like Milnor K-theory groups $\mathbf{K}_n^M(E)$, and in non-positive degrees like Grothendieck-Witt and Witt groups of quadratic forms $\operatorname{GW}(E)$ and $\operatorname{W}(E)$. The Milnor-Witt K-theory was used for solving some splitting problems for projective modules. For instance, generalizing ideas from the theory of Chow groups, one can use the Rost-Schmid complex of Morel to define the Chow-Witt groups $\CHt^*(X)$ for a smooth $k$-scheme $X$ (recalled in \ref{ClassicalChowWittGroups}), and the Euler class of a vector bundle of rank $r$ over $X$ (with a given trivialization of its determinant) as an element in $\CHt^r(X)$. When $X= \Spec A$ is a $2$-dimensional smooth affine variety and $r=2$, Barge and Morel proved in \cite{BargeMorel} that the Euler class associated to a projective module $P$ of rank $2$ over $A$ vanishes if and only if $P\simeq P'\oplus A$ for some projective module $P'$.

\begin{comment}

to define the Chow-Witt groups $\CHt^*(X)$ of, say, a $2$-dimensional smooth affine variety $X=\Spec(A)$ (see \ref{ClassicalChowWittGroups}). In \cite{BargeMorel}, Barge and Morel defined the Euler class $e(P)\in \CHt^2(X)$ associated to a projective module $P$ of rank $2$ over $A$ and they proved that this class vanishes if and only if $P\simeq P'\oplus A$ for some projective module $P'$. This result was later extended by Fasel in his thesis \cite{Fasel13} and other joint work with Srinivas \cite{FaselSrin}.

\end{comment}

\subsection*{Current and future work}
In this paper, we develop a conjectured\footnote{See \cite[Remark 2.6]{Rost96} and \cite[Remark 5.37]{Mor12}.} theory that studies general cycle complexes $C^*(X,M,\VV_X)$ and their (co)homology groups $A^*(X,M,\VV_X)$ (called {\em Chow-Witt groups with coefficients}) in a quadratic\footnote{Or rather: symmetric bilinear, since we allow fields of characteristic 2.} setting. The general coefficient systems $M$ for these complexes are called Milnor-Witt cycle modules. The main example of such a cycle module is given by the Milnor-Witt K-theory (see Theorem \ref{KMWisModule}); other examples can be deduced from Claim \ref{ThmDeg}, Example \ref{ExampleColimit} or Theorem \ref{AdjunctionTheorem} (e.g. the representability of hermitian K-theory in $\operatorname{SH}(k)$ will lead to a MW-cycle module, associated with hermitian  K-theory). A major difference with Rost's theory is that the grading to be considered is not $\ZZ$ but the category of virtual bundles (or, equivalently, the category of virtual vector spaces), where a virtual bundle $\VV$ is, roughly speaking, the data of an integer $n$ and a line bundle $\mathcal{L}$ (see Appendix \ref{VirtualObj}).
\par  Intuitively, Milnor-Witt cycle modules are given by (twisted) graded abelian groups equipped with extra data (restriction, corestriction, $\KMW$-action and residue maps). The difficult part was to find good axioms mimicking Rost's and taking into account the twists naturally arising in the non-oriented setting (virtual vector bundles play a major role). One important difference with the construction of Rost cycle modules \cite[Definition 1.1]{Rost96} is that we only need a weakened rule \ref{itm:R1c}: we do not consider multiplicities (we will show in future work that formulas involving multiplicities do hold as a consequence of our axioms). Moreover, we had to add another rule in our definition in order to link our theory with the classical one later (see \ref{itm:R4a} and Section \ref{Adjunction}).
\par  Following \cite[§3]{Rost96}, we define the necessary operations needed  on: pushforwards, pullbacks, multiplication with units and the motivic Hopf map $\eeta$, and boundary maps. Since our rules do not handle multiplicities, we first define pullbacks only for (essentially) smooth morphisms. We prove the usual functoriality theorems and show how to compute the (co)homology groups in special cases.

\par We can summarize the results as follows:
\begin{customthm}{2}(See Section \ref{FiveBasic} and Section \ref{Compatibilities})
For any scheme $X$, any virtual bundle $\VV_X$ and any Milnor-Witt cycle module $M$, there is a complex $C^*(X,M,\VV_X)$ equipped with pushforwards, pullbacks, a Milnor-Witt K-theory action and residue maps satisfying the standard functoriality properties.
\end{customthm} 
This formalism generalizes Schmid's thesis \cite{Schmid98} and provides more details to Morel's work \cite{Mor12}, regarding twists by virtual vector bundles.
\par An example of an application is as follows. Consider the localization sequence:
\begin{center}

$\CHt_p(Z) \to \CHt_p(X) \to \CHt_p(X\setminus Z)$,
\end{center}
where $Z\subset  X$ is a closed subvariety. Milnor-Witt cycle modules can be used to extend this exact sequence on the left and right, this is done in Section \ref{MWComplexe}.
\par In Section \ref{Acyclicity}, we prove a result similar to the Gersten Conjecture: the cohomology groups $A^p(X,M,\VV_X)$ are trivial when $X$ is a smooth semi-local scheme, $\VV_X$ a virtual bundle over $X$ and $p>0$. The proof is akin to Rost's Theorem 6.1 \cite{Rost96} and follows the classical ideas of Gabber and Panin.

% \par Another important problem is to define pullback maps (or {\em Gysin morphisms}) for regular closed immersion $Z\to X$ so that it corresponds geometrically to intersecting cycles with $Z$. In order to do this, we follow

\par Moreover, we prove the much expected homotopy invariance property, framing our work in $\AAA^1$-homotopy theory:
\begin{customthm}{3}(see Theorem \ref{HomotopyInvariance})
Let $X$ be a scheme, $V$ a vector bundle over $X$, $\pi:V\to X$ the canonical projection and $\VV_X$ a virtual vector bundle over $X$. Then, for every $q\in \ZZ$, the canonical morphism
\begin{center}
$\pi^*:A^q(X,M,\VV_X)\to A^q(V,M,-\TT_{V/X}+\VV_V)$

\end{center}
is an isomorphism.
\end{customthm}
Our proof of this theorem does not follow Rost's ideas. Indeed, in his original paper \cite[§7]{Rost96}, Rost constructs a cycle module $A^q[\rho,M]$ from any cycle module $M$ and any scheme morphism $\rho:Q\to B$, then uses a spectral sequence involving this new cycle module to prove the theorem. We could not easily extend the construction to our setting for some reasons related to twists. However, following ideas of Déglise (see \cite[§2]{Deg14}), we define a coniveau spectral sequence that helps us to reduce to the known case (see \ref{itm:(H)}). Note that we could also prove the result by giving a homotopy inverse on the complex level (as Rost did in \cite[§8]{Rost96}; see also \cite[Theorem 4.38]{Mor12}).

\par In \cite[§12]{Rost96}, Rost defines pullback maps $f^*$ (or {\em Gysin morphisms}) for morphisms $f:X\to Y$ between smooth schemes. In Section \ref{GysinMorphisms}, we do the same but for a more general set of maps, the classical local complete intersection (lci) morphisms, between possibly singular schemes. We prove functoriality and base change theorems which will prove to be useful later when considering multiplicities.
\par In Section \ref{ProductSection}, we construct products in our setting. Using the Gysin morphisms, we define then an intersection product on the homology groups of smooth schemes. In particular, we recover 	the intersection product already defined for Chow-Witt groups (see \cite[§3.4]{Fasel18bis}). We will use these notions to study the multiplicities naturally arising. Note however that we do not define a tensor product between Milnor-Witt cycle modules (but it can be obtained from Claim \ref{ThmDeg} below).

\par In future work, we will prove that Milnor-Witt cycle modules have a geometric interpretation in $\operatorname{SH}(k)$ as expected:

\begin{customclm}{4} \label{ThmDeg}
Let $k$ be a perfect field. The category of Milnor-Witt cycle modules is equivalent to the heart of the Morel-Voevodsky stable homotopy category (with respect to the homotopy t-structure):
\begin{center}

$ \mathfrak{M}^{MW}_k \simeq {\operatorname{SH}(k)}^\heartsuit$.
\end{center}
\end{customclm}
This result generalizes Theorem \ref{theseDeglise} and answers affirmatively an old conjecture of Morel (see \cite[Remark 2.49]{Mor12}). Note that we already have a description of the heart ${\operatorname{SH}(k)}^\heartsuit$ in terms of homotopy modules (which are strictly $\AAA^1$-invariant Nisnevich sheaves with an additional structure defined over the category of smooth schemes, see \cite[§5.2]{Mor03} or \cite[§1]{Deg10}). Cycle modules may be easier to work with since they are defined over fields. An important corollary of Claim \ref{ThmDeg}\footnote{This result seems closely related to the work of Bachmann and Yakerson. Indeed, in \cite{BachmannYakerson18}, the authors construct (for a strictly homotopy invariant sheaf $M$, a smooth $k$-scheme $X$ and a natural number $q$) a cycle complex $C^*(X,M,q)$ which may be compared to the complex defined in Section \ref{FiveBasic}. 
} is the following result (which was proved independently by Ananyevskiy and Neshitov in \cite[Theorem 8.12]{Neshitov2018}):
\begin{customcor}{5}
Let $k$ be a perfect field. The heart of the Morel-Voevodsky stable homotopy category is equivalent to the heart of the category of MW-motives \cite{DegFas18} (both equipped with their respective homotopy t-structures):
\begin{center}
${\operatorname{SH}(k)}^\heartsuit \simeq {\widetilde{\operatorname{DM}}(k)}^\heartsuit$.

\end{center}
\end{customcor}

In a forthcoming work, we will generalize the theory of Milnor-Witt cycle modules over $k$ to any base scheme $S$. We will also define general flat pullbacks and study the associated multiplicities.

\subsection*{Outline of the paper}

\par In Section \ref{MainExemple}, we recall some known
 properties of the Milnor-Witt K-theory (residue maps, specialization maps, transfer maps, homotopy invariance). Our main references are the original work of Morel \cite[§2]{Mor12} and also work of Calmès-Fasel (see \cite{Fasel18bis}). This will constitute our central example of a Milnor-Witt cycle module.

\par In Section \ref{Premodules} and Section \ref{Modules}, we define the main objects of our theory: Milnor-Witt cycle modules. These are Milnor-Witt cycle premodules satisfying two axioms \ref{itm:FD} and \ref{itm:C} which allow the associated complexes to have well-defined differentials.
\par In Section \ref{FiveBasic} and Section \ref{Compatibilities}, we define the basic operations needed further on: pushforwards, pullbacks, multiplication by units and the motivic Hopf map $\eeta$, and boundary maps. We prove the basic compatibility properties for these operations.
\par In Section \ref{MWComplexe}, we define homology groups called {\em Chow-Witt groups with coefficients} and describe how they are related to the classical Chow-Witt groups. We show how to compute these groups in some special cases in Section \ref{Acyclicity} and Section \ref{HomotopyInvarianceSection}.

\par In Section \ref{GysinMorphisms}, we define Gysin pullbacks for regular closed immersions and lci morphisms. We then proceed to construct products on the level of complexes in Section \ref{ProductSection}.
\par In Section \ref{Adjunction}, we define a pair of functors between our category of Milnor-Witt cycle modules and Rost's category of classical cycle modules. We can infer from Claim \ref{ThmDeg} that this is an adjunction but, for completeness, we give an elementary proof of this fact (conditionally to Claim \ref{StrongR1c} whose proof is postponed to another paper).
\par Finally, in Appendix \ref{VirtualObj}, we show how to define the category of virtual objects associated to an exact category and recall its basic properties (see \cite[§4]{Deligne87}).

\section*{Notation}\label{Conventions}
Throughout the paper, we fix a (commutative) field $k$ and we assume moreover that $k$ is perfect (of arbitrary characteristic). We consider only schemes that are noetherian and essentially of finite type\footnote{That is, isomorphic to a limit of finite type schemes with affine \'etale transition maps.} over $k$. All schemes and morphisms of schemes are defined over $k$.
\par By a field $E$ over $k$, we mean {\em a finitely generated extension of fields $E/k$}.
\par We denote by $\mathbb{F}_k$ the category of fields over $k$ (with obvious morphisms). 
\par If $A$ is a commutative ring, denote by $\mathbb{V}(A)$ the category of projective $A$-modules of finite type and $\mathfrak{V}(A)$ the category of virtual projective $A$-modules of finite type (see \cite[§4]{Deligne87} which uses the notation $\underline{K}(A)$; see also Appendix \ref{VirtualObj} for more details). Recall that there is a {\em contravariant} equivalence functor from $\mathbb{V}(A)$ to the category of vector bundles over $X=\Spec A$ (we follow the conventions of \cite[§9]{EGA1}). We will sometimes go from one category to the other without mentioning this functor.
\par Let $f:X\to S$ be a morphism of schemes and $\VV_S$ be a virtual bundle over $S$. We denote by $\VV_X$ or by $f^*\VV_S$  or by $\VV_S\times_S X$ the pullback of $\VV_S$ along $f$.
\par Now let $\mathfrak{F}_k$ be the category whose objects are couples $(E, \mathcal{V}_E)$ where $E$ is a field over $k$ and $\mathcal{V}_E\in \mathfrak{V}(E)$ is a virtual vector space (of finite dimension over $E$). A morphism $(E,\mathcal{V}_E)\to (F, \mathcal{V}_F)$ is the data of a morphism $E\to F$ of fields over $k$ and an isomorphism $\mathcal{V}_E \otimes_E F \simeq \mathcal{V}_F$ of virtual $F$-vector spaces\footnote{This category satisfies a universal property over $\mathbb{F}_k$ (one could generalize \cite[§4]{Deligne87} for fibred categories).}.
\par A morphism $(E,\mathcal{V}_E)\to (F, \mathcal{V}_F)$ in $\mathfrak{F}_k$ is said to be finite (resp. separable) if the field extension $F/E$ is finite (resp. separable).
\par Let $F/E$ be a field extension, we denote by $\Om_{F/E}$ the $F$-vector space of relative (Kähler) differentials. We use the same notation to denote its canonical image in the category of virtual vector spaces. Dually, let $X/S$ be a scheme morphism, we denote by $\TT_{X/S}$ the sheaf of modules of differentials.
\par Let $E$ be a field (over $k$) and $v$ a (discrete) valuation on $E$. We denote by $\mathcal{O}_v$ its valuation ring, by $\mathfrak{m}_v$ its maximal ideal and by $\kappa(v)$ its residue field. We consider only valuations on $E$ of {\em geometric type}, that is we assume: $k\subset \mathcal{O}_v$, the residue field $\kappa(v)$ is finitely generated over $k$ and satisfies $\operatorname{tr.deg}_k(\kappa(v))+1=\operatorname{tr.deg}_k(E)$ (in particular, $E$ cannot be $k$).
\par Let $E$ be a field and $v$ be a valuation on $E$. We denote by $\NN_v$ the $\kappa(v)$-vector space $\mathfrak{m}_v/\mathfrak{m}_v^2$ and call it the {\em normal bundle of $v$}.
\par For $E$ a field (resp. $X$ a scheme), denote by $\AAA^1_E$ (resp. $\AAA^1_X$) the virtual vector space of dimension $1$ over $E$ (resp. the virtual affine space of rank one over $X$).

\section*{Acknowledgment}
I deeply thank my two PhD advisors Frédéric D\'eglise and Jean Fasel. I would also like to thank the referee for their many insightful comments and suggestions. This work received support from the French "Investissements d'Avenir" program, project ISITE-BFC (contract ANR-lS-IDEX-OOOB).

\section{Main Example: the Milnor-Witt K-theory}\label{MainExemple}

We describe the Milnor-Witt K-theory, as defined by Morel (see \cite[§3]{Mor12} or \cite[§1.1]{Fasel18bis}).

\begin{Def}\label{kMWdefinition}

Let $E$ be a field. The Milnor-Witt K-theory algebra of $E$ is defined to be the quotient of the free $\ZZ$-graded algebra generated by the symbols $[a]$ of degree $1$ for any $a\in E^\times$ and a symbol $\eeta$ in degree $-1$ by the following relations:
\begin{itemize}
\item $[a][1-a]=0$ for any $a\in E^{\times}\setminus\{1\}$.
\item $[ab]=[a]+[b]+\eeta[a][b]$ for any $a,b\in E^\times$.
\item $\eeta[a]=[a]\eeta$ for any $a\in E^\times$.
\item $\eeta(\eeta[-1]+2)=0$.
\end{itemize}
The relations being homogeneous, the resultant algebra is $\ZZ$-graded. We denote it by $\kMW_*(E)$.
\end{Def}

\begin{Par}
{\sc Notation} \label{NotationsKMW}  We will use the following notations.
\begin{itemize}
\item $[a_1,\dots, a_n]=[a_1]\dots [a_n]$ for any $a_1,\dots, a_n\in E^\times$.
\item $\langle a \rangle = 1 + \eeta[a]$ for any $a\in E^\times$.
% \item $h=2+\eeta[-1]=1+\langle -1 \rangle$.
\item $\epsilon=-\langle-1 \rangle$.
\item $n_\epsilon=\sum_{i=1}^n \langle (-1)^{i-1}\rangle$ for any $n\geq 0$, and $n_\epsilon=\epsilon(-n)_\epsilon$ if $n<0$.
\end{itemize}
\end{Par}

\begin{Par}{\sc Twisted Milnor-Witt K-theory} \label{TwistedMWKTheory}
Let $E$ be a field and $\VV_E$ a virtual vector bundle over $E$ with rank $n$ and determinant $\mathcal{L}_E$. The group $E^\times$ of invertible elements of $E$ acts naturally on $\mathcal{L}_E^\times$, the set of non-zero elements in $\mathcal{L}_E$ ; hence the free abelian group $\ZZ[\mathcal{L}^{\times}_E]$ is a $\ZZ[E^\times]$-module. Define
\begin{center}

$\KMW(E,\VV_E)=\kMW_{n}(E)\otimes_{\ZZ[E^\times]} \ZZ[\mathcal{L}_E^\times]$.
\end{center}

By abuse of notation, we might denote by $\eeta$ the element $(\eeta\otimes 1)\in \KMW(E,-\AAA^1)$ and by $[u]$ the element $([u]\otimes 1)\in \KMW(E,\AAA^1)$ (where $u\in E^\times$).
\par 
Let $\VV_E$ and $\mathcal{W}_E$ be two virtual bundles over $E$. The product of the Milnor-Witt K-theory groups induces a product
\begin{center}

$\KMW(E,\mathcal{V}_E)\otimes \KMW(E,\mathcal{W}_E)\to \KMW(E,\mathcal{V}_E+\mathcal{W}_E)$
\\ $(x\otimes l,x'\otimes l') \mapsto (xx')\otimes (l\wedge l')$,
\end{center}
so that $\KMW(E,-)$ is a lax monoidal functor from the category of virtual bundles over $E$ to the category of abelian groups.
\end{Par}
\begin{Par}{\sc Residue morphisms}  \label{KMWresidue} (see \cite[Theorem 3.15]{Mor12}) Let $E$ be a field endowed with a discrete valuation $v$. We choose a uniformizing parameter $\pi$. As in the classical Milnor K-theory, we can define a residue morphism
\begin{center}

$\partial^{\pi}_v:\kMW_*(E)\to \kMW_{*-1}(\kappa(v))$
\end{center}
commuting with the multiplication by $\eeta$ and satisfying the following two properties:
\begin{itemize}
\item $\partial^{\pi}_v([\pi,a_1,\dots, a_n])=[\overline{a_1},\dots, \overline{a_n}]$ for any $a_1,\dots,a_n\in \mathcal{O}_v^{\times}$.
\item $\partial^{\pi}_v([a_1,\dots, a_n])=0$ for any $a_1,\dots,a_n\in \mathcal{O}_v^{\times}$.
\end{itemize}
This morphism does depend on the choice of $\pi$. Indeed, if we consider another uniformizer $\pi'$ and write $\pi'=u\pi$ where $u$ is a unit, then we have $\partial^{\pi}_v(x)=\langle u \rangle \partial^{\pi'}_v(x)$ for any $x\in \kMW_*(E)$. Nevertheless, by making a choice of isomorphism $\NN_v\simeq \AAA^1_{\kappa(v)}$ (where $\NN_v=\mathfrak{m}_v/\mathfrak{m}_v^2$ is the normal cone of $v$), we can define a twisted residue morphism that does not depend on $\pi$:
\begin{center}

$\partial_v:\KMW(E,\mathcal{V}_E)\to \KMW(\kappa(v),-\NN_v+ \VV_{\kappa(v)})$
\\ $x\otimes l \mapsto \partial^{\pi}_v(x) \otimes (\bar{\pi}^*\wedge l)$
\end{center}
where $\VV$ is a virtual vector bundle over $\mathcal{O}_v$, $\bar{\pi}$ is the canonical projection of $\pi$ modulo $\mathfrak{m}_v$ and $\bar{\pi}^*$ its canonical associated linear form (see also Example \ref{ExGradedLines}).
\par We have the following proposition (see \cite[Lemma 3.19]{Mor12}).
%\cite[Lemma 1.20]{CalFasel18}; .

\end{Par}
\begin{Pro} \label{KMWRuleR3a}

Let $j:E\subset F$ be a field extension, and let $w$ be a discrete valuation on $F$ which restricts to a discrete valuation $v$ on $E$ with ramification index $e$. Let $i:\kappa(v)\to \kappa(w)$ be the field extension of the residue fields. Let $\VV$ be a virtual vector bundle over $\mathcal{O}_v$. Then, we have a commutative diagram
\begin{center}
$\xymatrix{
\KMW(E,\VV_E) \ar[r]^-{\partial_v} \ar[d]_{j_*} & \KMW(\kappa(v),-\NN_v+\VV_{\kappa(v)}) \ar[d]^{e_\epsilon\cdot i_*} \\
\KMW(F,\VV_F) \ar[r]^-{\partial_w} & \KMW(\kappa(w),-\NN_w+\VV_{\kappa(w)}) 
}$
\end{center}
where the right vertical map ${e_\epsilon\cdot i_*}$ is defined by
\begin{center}

$x\otimes \bar{\rho}^* \otimes l \mapsto e_\epsilon i(x)\otimes \bar{\pi}^* \otimes l$
\end{center}
where $\pi$ is a uniformizer for $w$ and $\rho=\pi^e$ is a uniformizer for $v$.
\end{Pro}

\begin{Par}{\sc Specialization map} (see \cite[Lemma 3.16]{Mor12})
Keeping the previous notations, we also have a specialization map
\begin{center}

$s^{\pi}_v:\kMW_*(E)\to \kMW_*(\kappa(v))$
\end{center}
commuting with the multiplication by $\eeta$ and satisfying
\begin{center}
$s^{\pi}_v([\pi^{m_1}a_1,\dots,\pi^{m_n}a_n])=[\overline{a_1},\dots, \overline{a_n}]$
\end{center}
 for any $a_1,\dots, a_n\in \mathcal{O}_v^{\times}$ and any integers $m_1,\dots, m_n$. If $\pi'$ is another uniformizer, write $\pi'=\pi u$ where $u$ is a unit. Then, for any $x\in \kMW_*(E)$, we have $s^{\pi'}_v(x)=s^{\pi}_v(x)-[\bar{u}]\partial^{\pi'}_v(x)$.
\end{Par}
It is related with the residue map as follows:
\begin{center}
$s^{\pi}_v(\alpha)=\langle - 1 \rangle \partial^{\pi}_v([-\pi]\cdot \alpha)$

\end{center}  for any $\alpha \in \kMW_*(E)$. One can define a twisted specialization map (also denoted by $s^\pi_v$) as follows:
\begin{center}

$s^{\pi}_v:\KMW(E,\VV_E)\to \KMW(E,\VV_E)$
\\ $x\otimes l \mapsto s^{\pi}_v(x)\otimes l$.

\end{center}
It satisfies:
\begin{center}

$s^{\pi}_v=\Theta_\pi \circ \partial_v \circ \gamma_{[-\pi]}$
\end{center}
	where $\gamma_{[-\pi]}:\KMW(E,\VV_E)\to \KMW(E,\AAA_E^1+\VV_E)$ is the multiplication by $([-\pi]\otimes 1)$ and where $\Theta_\pi$ is induced by the composite of two isomorphisms $-\NN_v+\AAA^1_E\simeq \AAA^1_E+(-\NN_v)\simeq 0$ of virtual vector bundles (the first one is the canonical switch isomorphism; the second one comes from the choice of uniformizer $\pi$).

\begin{Par}{\sc The homotopy exact sequence} \label{KMWFD} Let $E$ be a field with a virtual vector bundle $\VV_E$, and let $E(t)$ be the function field of $E$ in the variable $t$. Any point $x$ of codimension 1 in $\AAA^1_E$ corresponds to a maximal ideal $\mathfrak{m}_x$ in $E[t]$, generated by a unique irreducible monic polynomial $\pi_x$ which is a uniformizing parameter of the $\mathfrak{m}_x$-adic valuation $v_x$ on $E(t)$. Additionally, the valuation $v_\infty$ at $\infty$ is given by $v_\infty(f/g)=\deg(g)-\deg(f)$ and we can choose (as in \cite[§4.2]{Mor12}) the rational function $-1/t$ as uniformizing parameter. For any $a \in E(t)^\times$, there is only a finite number of $x$ such that $v_x(a)\neq 0$. Furthermore, since the Milnor-Witt K-theory of $E$ is generated by elements of degree $1$ (and $\eeta$ in degree $-1$), it follows that the total residue map
\begin{center}

$d:\KMW(E(t),\Om_{E(t)/k}+ \VV_{E(t)})\to \bigoplus_{x\in (\AAA^1_E)^{(1)}} \KMW(E(x),\Om_{E(x)/k}+\VV_{E(x)})$
\end{center}
with $d=\sum_x \Theta_x\circ \partial_{v_x}$ is well-defined (where $\Theta_x$ comes from the canonical isomorphism $-\NN_{v_x}+\Om_{E(t)/k}\otimes_{E(t)}E(x)\simeq \Om_{E(x)/k}$ of virtual vector bundles).
The following theorem (see \cite[Theorem 3.24]{Mor12} for a non-twisted statement) is one of the most fundamental results of Milnor-Witt K-theory. Let ${i_*:\KMW(E,\AAA^1_E+\Om_{E/k}+\VV_E)\to \KMW(E(t),\Om_{E(t)/k}+\VV_{E(t)}})$ be the morphism induced by the field extension $E\subset E(t)$ and the canonical isomorphism $\AAA^1_{E(t)}+\Om_{E/k}\otimes_E E(t) \simeq \Om_{E(t)}$ of virtual vector bundles.
\end{Par}
\begin{The}[Homotopy invariance] 

 \label{KMWhmtpyInv} 
With the previous notations, the following sequence is split exact

\begin{center}

% \resizebox{18cm}{!}{}
% $\xymatrix@C=1pt@R=20pt{
$\xymatrix@C=10pt@R=20pt{
0 \ar[r] &  \KMW(E,\AAA^1_E+\Om_{E/k}+\VV_E) \ar[r]^-{i_*}   & \KMW (E(t),\Om_{E(t)/k}+ \VV_{E(t)}) \ar[r]^-d  &  \bigoplus_{x\in {(\AAA_E	^1)}^{(1)}} \KMW(\kappa(x), \Om_{\kappa(x)/k}+\VV_{\kappa(x)}) \ar[r] & 0.
}
$

\end{center}
\begin{comment}

$
\xymatrix{
0 \ar[r] &  \KMW(E,\AAA^1_E+\Om_{E/k}+\VV_E) \ar[r]^-{i_*}   & \KMW (E(t),\Om_{E(t)/k}+ \VV_{E(t)}) \\
  & \, \, \,\, \, \,\, \, \,\, \, \,\, \, \,\, \, \,\, \, \,\, \, \, \, \, \,\ar[r]^-d  &  \bigoplus_{x\in {(\AAA_E	^1)}^{(1)}} \KMW(\kappa(x), \Om_{\kappa(x)/k}+\VV_{\kappa(x)}) \ar[r] & 0.
}
$
\end{comment}

%  
%  \resizebox{17cm}{!}{
%  
%  
%  $
%  \xymatrix{
%  0 \to \KMW(E,\AAA^1_E+\Om_{E/k}+\VV_E) \ar[r]^-{i_*}   & \KMW (E(t),\Om_{E(t)/k}+ \VV_{E(t)}) \ar[r]^-d  &  \bigoplus_{x\in {(\AAA_E	^1)}^{(1)}} \KMW(\kappa(x), \Om_{\kappa(x)/k}+\VV_{\kappa(x)}) \ar[r] & 0.
%  }
%  $}

\end{The}
\begin{proof}
See \cite[Theorem 3.24]{Mor12}, \cite[Proposition 1.20]{Fasel18bis}.
\end{proof}

\begin{Par}{\sc Transfer maps}
\label{kMWTransfert}
Let $\phi:E\to F$ be a monogeneous finite field extension with $\VV_E$ a virtual vector bundle over $E$ and choose $x\in F$ such that $F=E(x)$. The homotopy exact sequence implies that for any $\beta\in \KMW(F,\Om_{F/k}+\VV_{F})$ there exists ${\gamma\in \KMW(E(t),\Omega_{E(t)/k}+\VV_{E(t)})}$ with the property that $d(\gamma)=\beta$. Now the valuation at $\infty$ yields a morphism
  ${\partial_\infty^{-1/t}:\KMW(E(t),\Omega_{E(t)/k}+\VV_{E(t)})\to
   \KMW(E,\Omega_{E/k}+\VV_E)}$ which vanishes on the image of $i_*$. We denote by $\phi^*(\beta)$ or by $\operatorname{Tr}^F_E(\beta)$ the element $-\partial^{-1/t}(\gamma)$; it does not depend on the choice of $\gamma$. This defines a group morphism
  \begin{center}
  
  $\phi^*:\KMW(F,\Omega_{F/k}+\VV_F)\to \KMW(E,\Omega_{E/k}+\VV_E)$
  \end{center}
  called the {\em transfer map} and also denoted by $\operatorname{Tr}^F_E$. The following result completely characterizes the transfer maps.

\end{Par}
\begin{Lem} Keeping the previous notations, let
  \begin{center}
  
  $d:\KMW(E(t),\Omega_{F(t)/k}+\VV_{E(t)})\to
   (\bigoplus_x\KMW(E(x),\Omega_{E(x)/k}+\VV_{E(x)}))\oplus \KMW(E,\Omega_{E/k}+\VV_E)$
  \end{center}
  be the total twisted residue morphism (where $x$ runs through the set of monic irreducible polynomials in $E(t)$). Then, the transfer maps $\operatorname{Tr}^{E(x)}_E$ are the unique morphisms such that $\sum_x ( \operatorname{Tr}^{E(x)}_E\circ d_x)+d_\infty=0$.
  \end{Lem}
  \begin{proof}
  
  See \cite[§4.2]{Mor12}.
% or \cite[Lemma 2.10]{CalFasel18}.
  \end{proof}
  Let $\phi:E\to F$ be a finite field extension. We can find a factorization
  \begin{center}
  
  $E=F_0\subset F_1\subset \dots \subset F_n=F$
  \end{center}
  such that $F_i/F_{i-1}$ is monogeneous for any $i=1,\dots,n$. We set $\operatorname{Tr}^F_E=\operatorname{Tr}^{F_1}_E\circ \dots \circ \operatorname{Tr}^F_{F_{n-1}}$. Morel proved (see \cite[Theorem 4.27]{Mor12}) that the definition does not depend on the choice of the factorization (thanks to the twists).

%           (D1),(D2),(D3),(D4) ok.
%           (R1a),(R1b) ok
%           (R2a) ok
%           (R2b),(R2c) ok \cite{CF17}
%           (R3d) ok
%           (R3e) ok
%           (R3a) ok
%           (R3c) ok
%           (R3e) ok

The pullback $\operatorname{Tr}^F_E$ is functorial by definition. It satisfies the usual projection formulae and the following base change theorem.
%  (see \cite[Lemma 2.12]{CalFasel18}) 

\begin{Pro}[Base change] \label{KMWR1c}

 Let $\phi:(E,\VV_E)\to (F,\VV_F)$ and $\psi:(E,\VV_E)\to (L,\VV_L)$ with $\phi$ finite and $\psi$ separable. Let $R$ be the ring $F\otimes_E L$. For each $p\in \Spec R$, let $\phi_p:(L,\VV_L)\to (R/p, \VV_{R/p})$ and $\psi_p:(F,\VV_F)\to (R/p, \VV_{R/p})$ be the morphisms induced by $\phi$ and $\psi$. One has
\begin{center}
$\psi_*\circ \phi^*=\displaystyle \sum_{p\in \Spec R} (\phi_p)^*\circ \Theta_p \circ (\psi_p)_*$
\end{center}
where $\Theta_p$ is induced by the canonical isomorphism $\Om_{F/E}\otimes_F (R/p)\simeq \Om_{(R/p)/L}$ (since $\psi$ is separable).
\end{Pro}
\begin{proof}

This follows from general base change theorem on the Milnor-Witt K-theory (see \cite[Corollaire 12.3.7]{Fasel13}). 
\end{proof}
Finally, transfer maps are compatible with residue maps in the following sense.
\begin{Pro} \label{KMWRuleR3b}
 Let $E\to F$ be a finite extension of fields over $k$, let $v$ be a valuation on $E$ and let $\VV$ be a virtual $\mathcal{O}_v$-module. For each extension $w$ of $v$, we denote by ${\phi_w: (\kappa(v), \VV_{\kappa(v)})} \to ( \kappa(w), \VV_{\kappa(w)})$ the morphism induced by $\phi:(E,\VV_E)\to (F,\VV_F)$. We have:
\begin{center}
$\partial_v \circ \phi^*= \sum_w \phi_w^* \circ \partial_w$.
\end{center}

\end{Pro}
\begin{proof}
This follows from a more general result of Morel: transfer maps induce morphims of the corresponding Rost-Schmid complexes \cite[§5]{Mor12}.\footnote{The reader looking for a more elementary proof of the proposition may try to follow the ideas of \cite{Suslin80}.} 
\end{proof}
\section{Milnor-Witt Cycle Premodules} \label{Premodules}
%          \begin{Def}
%          We define the functor $\KMW:\mathfrak{F}_k\to \mathfrak{A}b$ by
%          \begin{center}
%          $\KMW= \underline{\widehat{\pi}}_0(\mathbb{S}_0)$.
%          \end{center}
%          \end{Def}

%          \begin{Pro}
%          Let $E$ be a field over $k$, $\mathcal{L}$ be a virtual line bundle over $E$ and $n$ be an integer. We have a morphisms of abelian groups 
%          \begin{center}
%          $K^{MW}_n(E,\mathcal{L}) \simeq \KMW(E, \mathcal{L}+(n-1)\cdot \mathbb{A}^1)$.
%          \end{center}
%          \end{Pro}

We now define the main object of this paper.
%     In order to make the definition more digest, we split it into several ones \footnote{Je vais sûrement commencer par définir un "MW cycle group" comme étant un MW-prémodule sans action de $\KMW$, puis définir un MW-prémodule comme étant un "MW-group" avec un pairing de $\KMW$. Pour l'instant, il y a juste un souci avec la règle R3a qui va me demander quelques vérifications. Avec un peu de chance, je peux ranger la règle R3a dans la définition de "MW group" (i.e. l'énoncer sans ramification) puis démontrer plus tard la règle R3a avec ramification}. 
See the subsequent remarks for more details.

\begin{Def} \label{defMWmodules}
A Milnor-Witt cycle premodule $M$ (also written: MW-cycle premodule) is a functor from $\mathfrak{F}_k$ to the category $\Ab$ of abelian groups with the following data \ref{itm:D1},\dots, \ref{itm:D4} and the following rules \ref{itm:R1a},\dots, \ref{itm:R4a}.
\begin{description}
\item[\namedlabel{itm:D1}{D1}] Let $\phi : (E,\mathcal{V}_E)\to (F, \mathcal{V}_F)$ be a morphism in $\mathfrak{F}_k$. The functor $M$ gives a morphism $\phi_*:M(E,\mathcal{V}_E) \to M(F,\mathcal{V}_F)$.
\item [\namedlabel{itm:D2}{D2}] Let $\phi : (E,\mathcal{V}_E)\to (F, \mathcal{V}_F)$ be a morphism in $\mathfrak{F}_k$ where the morphism $E\to F$ is {\em finite}. There is a morphism  $\phi^*:M(F,\Om_{F/k}+\mathcal{V}_F) \to M(E,\Om_{E/k}+\mathcal{V}_E)$.
\item [\namedlabel{itm:D3}{D3}] Let $(E,\mathcal{V}_E)$ and $(E,\mathcal{W}_E)$ be two objects of $\mathfrak{F}_k$. For any element $x$  of $\KMW(E,\mathcal{W}_E)$, there is a morphism 
\begin{center}
$\gamma_x : M(E,\mathcal{V}_E)\to M(E,\mathcal{W}_E+\mathcal{V}_E)$
\end{center}
so that the functor $M(E,-):\mathfrak{V}(E)\to \Ab$ is a left module over the lax monoidal functor $\KMW(E,-):\mathfrak{V}(E)\to \Ab$ (see \cite[Definition 39]{Yetter03}; see also remarks below).

\item [\namedlabel{itm:D4}{D4}] Let $E$ be a field over $k$, let $v$ be a valuation on $E$ and let $\mathcal{V}$ be a virtual projective {$\mathcal{O}_v$-module} of finite type. Denote by $\mathcal{V}_E=\VV \otimes_{\mathcal{O}_v} E$ and $\VV_{\kappa(v)}=\VV \otimes_{\mathcal{O}_v} \kappa(v)$. There is a morphism
\begin{center}
$\partial_v : M(E,\VV_E) \to M(\kappa(v), - \NN_v+\VV_{\kappa(v)}).$
\end{center}
\item [\namedlabel{itm:R1a}{R1a}] Let $\phi$ and $\psi$ be two composable morphisms in $\mathfrak{F}_k$. One has
\begin{center}
 $(\psi\circ \phi)_*=\psi_*\circ \phi_*$.
\end{center}
\item [\namedlabel{itm:R1b}{R1b}]  Let $\phi$ and $\psi$ be two composable finite morphisms in $\mathfrak{F}_k$. One has
\begin{center}
 $(\psi\circ \phi)^*=\phi^*\circ \psi^*$.
\end{center}
\item [\namedlabel{itm:R1c}{R1c}] Consider $\phi:(E,\VV_E)\to (F,\VV_F)$ and $\psi:(E,\VV_E)\to (L,\VV_L)$ with $\phi$ finite and $\psi$ separable. Let $R$ be the ring $F\otimes_E L$. For each $p\in \Spec R$, let $\phi_p:(L,\VV_L)\to (R/p, \VV_{R/p})$ and $\psi_p:(F,\VV_F)\to (R/p, \VV_{R/p})$ be the morphisms induced by $\phi$ and $\psi$. One has
\begin{center}
$\psi_*\circ \phi^*=\displaystyle \sum_{p\in \Spec R} (\phi_p)^*\circ (\psi_p)_*$.
\end{center}

\item [\namedlabel{itm:R2}{R2}] Let $\phi : (E,\VV_E)\to (F,\VV_F)$ be a morphism in $\mathfrak{F}_k$, let $x$ be in $\KMW (E,\mathcal{W}_E)$ and $y$ be in $\KMW (F,\Om_{F/k}+\mathcal{W'}_F)$ where $(E,\mathcal{W}_E)$ and $(F,\mathcal{W}'_F)$ are two objects of $\mathfrak{F}_k$.
\item [\namedlabel{itm:R2a}{R2a}] We have $\phi_* \circ \gamma_x= \gamma_{\phi_*(x)}\circ \phi_*$.
\item [\namedlabel{itm:R2b}{R2b}] Suppose $\phi$ finite. We have $\phi^*\circ \gamma_{\phi_*(x)}=\gamma_x \circ \phi^*$.
\item [\namedlabel{itm:R2c}{R2c}] Suppose $\phi$ finite. We have $\phi^*\circ \gamma_y \circ \phi_*= \gamma_{\phi^*(y)}$.
\item [\namedlabel{itm:R3a}{R3a}] 
  Let $E\to F$ be a field extension and $w$ be a valuation on $F$ which restricts to a non trivial valuation $v$ on $E$ with ramification $e$. Let $\VV$ be a virtual $\mathcal{O}_v$-module so that we have a morphism $\phi:(E, \VV_E)\to (F,\VV_F)$ which induces a morphism
 ${\overline{\phi}:(\kappa(v),-\NN_v+\VV_{\kappa(v)})\to ( \kappa(w),-\mathcal{N}_w+\VV_{\kappa(w)}})$. We have 
 
 \begin{center}
 
$\partial_w \circ \phi_*=\gamma_{e_\epsilon} \circ \overline{\phi}_* \circ \partial_v$. 
 \end{center}
\item [\namedlabel{itm:R3b}{R3b}] Let $E\to F$ be a finite extension of fields over $k$, let $v$ be a valuation on $E$ and let $\VV$ be a $\mathcal{O}_v$-module. For each extension $w$ of $v$, we denote by ${\phi_w: (\kappa(v), \VV_{\kappa(v)}) \to ( \kappa(w), \VV_{\kappa(w)}})$ the morphism induced by ${\phi:(E,\VV_E)\to (F,\VV_F)}$. We have
\begin{center}
$\partial_v \circ \phi^*= \sum_w (\phi_w)^* \circ \partial_w$.
\end{center}
\item [\namedlabel{itm:R3c}{R3c}] Let $\phi : (E,\VV_E)\to (F,\VV_F)$ be a morphism in $\mathfrak{F}_k$ and let $w$ be a valuation on $F$ which restricts to the trivial valuation on $E$. Then 
\begin{center}
$\partial_w \circ \phi_* =0$.
\end{center}
\item [\namedlabel{itm:R3d}{R3d}] Let $\phi$ and $w$ be as in \ref{itm:R3c}, and let $\overline{\phi}:(E,\VV_E)\to (\kappa(w),\VV_{\kappa(w)})$ be the induced morphism. For any uniformizer $\pi$ of $v$, we have
\begin{center}
$\partial_w \circ \gamma_{[-\pi]}\circ \phi_*= \overline{\phi}_*$.
\end{center}

\item [\namedlabel{itm:R3e}{R3e}] Let $E$ be a field over $k$, $v$ be a valuation on $E$ and $u$ be a unit of $v$. Then
\begin{center}
$\partial_v \circ \gamma_{[u]}=\gamma_{\epsilon[\overline{u}]} \circ \partial_v$ and
\\ $\partial_v \circ \gamma_\eta =\gamma_\eta \circ \partial_v$.
\end{center}

 \item [\namedlabel{itm:R4a}{R4a}] Let $(E,\VV_E)\in \mathfrak{F}_k$ and let $\Theta$ be an endomorphism of $(E,\VV_E)$ (that is, an automorphism of $\VV_E$). Denote by $\Delta$ the canonical map\footnote{See \ref{RemarkOnR4a} for more details.} from the group of automorphisms of $\VV_E$ to the group $\kMW(E,0)$. Then
 \begin{center}
 
 $\Theta_*=\gamma_{\Delta(\Theta)}:M(E,\VV_E)\to M(E,\VV_E)$.
 \end{center}
 
\end{description}
\end{Def}

Here is a series of remarks about our definition.
\begin{Par}
One could expand the definition and work with coefficients in an arbitrary abelian category instead of $\mathcal{A}b$.
\end{Par}
\begin{Par}\label{RmkDefMWmodules1}

 The maps $\phi_*$ and $\phi^*$ are respectively called the restriction and corestriction morphisms. We also use the notations $\phi_*= \res_{F/E}$ and $\phi^*=\cores_{F/E}$.
 \par The map $\partial_v$ is called the residue map.

\end{Par}
\begin{Par} \label{RmkDefMWmodules2}

 \ref{itm:D1} and \ref{itm:R1a} are redundant, given the fact that $M$ is a (covariant) functor.

\end{Par}

\begin{Par} \label{RingStructure} Intuitively, the data \ref{itm:D3} state that  we have a left $\bigoplus_{\mathcal{W}_E\in \mathfrak{V}(E)}\KMW(E,\mathcal{W}_E)$-module structure on the abelian group
\begin{center}
$\bigoplus_{\VV_E\in \mathfrak{V}(E)} M(E, \VV_E)$.
\end{center}
This statement is inaccurate for the following reasons. First, one should avoid the use of infinite direct sums (our theory ought to work if we replace $\Ab$ by any abelian category). Moreover, we have implied that $$\bigoplus_{\mathcal{W}_E\in \mathfrak{V}(E)}\KMW(E,\mathcal{W}_E)$$
is a ring which is, rigorously speaking, not true. Among other reasons, the multiplicative structure on this abelian group does not have an identity element because we do not have strict equality between a virtual vector bundle $\mathcal{W}_E$ and $0+\mathcal{W}_E$ (there is only a canonical isomorphism). Instead, we should say that $\KMW(E,-):\mathfrak{V}(E)\to \Ab$ is a lax monoidal functor (see \cite[Definition 5]{Yetter03}) and that $M(E,-):\mathfrak{V}(E)\to \Ab$ is a left module over $\KMW(E,-)$ (see \cite[Definition 39]{Yetter03}). 
\par Finally, we add that the functor $\KMW(E,-)$ is skewed or antisymetric in the sense that for any virtual vector bundles $\VV_E$ and $\mathcal{W}_E$ of rank $m,n$ (respectively), and for any  $x\in \KMW(E,\VV_E)$ and $y\in \KMW(E,\mathcal{W}_E)$, we have $xy=(-1)^{nm}\Theta (yx)$ where $\Theta$ is induced by the canonical switch isomorphism $\mathcal{W}_E+\VV_E\simeq \VV_E+\mathcal{W}_E$. Intuitively, the switch isomorphism $\Theta$ corresponds to multiplication by $\langle -1 \rangle^{nm} = (-\epsilon)^{nm}$ hence we say that the abelian group equipped with the obvious multiplication maps

\begin{center}
$\bigoplus_{\mathcal{W}_E\in \mathfrak{V}(E)}\KMW(E,\mathcal{W}_E)$ 
\end{center}
is $\epsilon$-commutative.
\end{Par}

\begin{Par}
A MW-cycle premodule $M$ can be equipped with a right $\KMW$-module action as follows. For $x\in \KMW(E,\mathcal{W}_E)$, define 
\begin{center}

$\rho_x:M(E,\VV_E)\to M(E,\VV_E+\mathcal{W}_E)$ 
\end{center}
to be the composite $(-1)^{nm}\Theta \circ \gamma_x$ where $\Theta$ is induced by the canonical switch isomorphism $\VV_E+\mathcal{W}_E\simeq \mathcal{W}_E+\VV_E$ and where $m,n$ are the ranks of $\mathcal{W}_E,\VV_E$, respectively.
\end{Par}

\begin{Par}\label{RmkDefMWmodules5}

 With the notations of \ref{itm:D4}, note that we have a short exact sequence of $\kappa(v)$-vector space
\begin{center}
$0\to \NN_v \to \Om_{\mathcal{O}_v/k}\otimes_{\mathcal{O}_v} \kappa(v) \to \Om_{ \kappa(v)/k} \to 0$
\end{center}
and that we have a canonical isomorphism
\begin{center}
$\Om_{\mathcal{O}_v/k}\otimes_{\mathcal{O}_v} \kappa(v) \simeq \Om_{E/k}\otimes_{\mathcal{O}_v} \kappa(v)$
\end{center}
so that the data \ref{itm:D4} is equivalent to having a map
\begin{center}
$\delta_v:M(E,  \Om_{E/k}+\VV_E)\to 
M(\kappa(v),\Om_{ \kappa(v)/k}+\VV_{ \kappa(v)}
)$.
\end{center}
	\end{Par}
	\begin{Par}
	We assume that \ref{itm:R1b} also contains the rule $\phi^*=\Id$ when $\phi$ is the identity map.

	\end{Par} 

\begin{Par}

In \ref{itm:R1c}, the fact that the extension is separable means that there are no multiplicities to consider. This rule is weaker than the corresponding one used by Rost (see \cite[Definition 1.1.(R1c)]{Rost96}) but it is sufficient for our needs. Indeed, we will (in a future paper) prove the following theorem (see also Theorem \ref{BaseChangelci}).
\end{Par}
\begin{Claim}[Strong R1c] \label{StrongR1c} Let $\phi:(E,\VV_E)\to (F,\VV_F)$ and $\psi:(E,\VV_E)\to (L,\VV_L)$ with $\phi$ finite. Let $R$ be the ring $F\otimes_E L$. For each $p\in \Spec R$, let $\phi_p:(L,\VV_L)\to (R/p, \VV_{R/p})$ and ${\psi_p:(F,\VV_F)\to (R/p, \VV_{R/p})}$ be the morphisms induced by $\phi$ and $\psi$. One has
\begin{center}
$\psi_*\circ \phi^*=\displaystyle \sum_{p\in \Spec R} m_p \cdot (\phi_p)^*\circ (\psi_p)_*$
\end{center}

where $m_p$ are some quadratic forms.

\end{Claim}
The multiplicities $m_p$ remain to be defined. One expected property is that the rank of $m_p$ is the usual multiplicity as defined in \cite[§1]{Rost96} (that is, the length of the localized ring $R_{(p)}$).

\begin{Par}
 In the rules \ref{itm:R2a}, \ref{itm:R2b} and \ref{itm:R2c}, $\phi_*(x)$ and $\phi^*(y)$ are defined as in \ref{kMWTransfert}. Moreover, these three rules are sometimes called {\em projection formulae} and could be illustrated with the following identities:
\begin{center}
$\phi_*(x \cdot \rho)=\phi_*(x)\cdot \phi_*(\rho)$,
\\ $\phi^*(\phi_*(x)\cdot \mu)=x\cdot \phi^*(\mu)$,
\\ $\phi^*(y\cdot \phi_*(\rho))=\phi^*(y)\cdot \rho$.
\end{center}
Rigorously, we should write \ref{itm:R2b} as $\Theta\circ \phi^*\circ \Theta' \circ \gamma_{\phi_*(x)}=\gamma_x\circ \phi^*$ where $\Theta,\Theta'$ are induced by the canonical isomorphisms $\mathcal{W}_F+\Om_{F/k}\simeq \Om_{F/k}+\mathcal{W}_F$ and $\Om_{E/k}+\mathcal{W}_E \simeq \mathcal{W}_E+\Om_{E/k}$, respectively.

\end{Par}

\begin{Par}
With the notation of \ref{itm:R3a}, notice that we do have an isomorphism $\NN_v \otimes_{\kappa(v)} \kappa(w) \simeq \NN_w$ of virtual vector spaces defined as in \ref{KMWRuleR3a}. Moreover, we consider $e_\epsilon$ to be an element of $\KMW(\kappa(w),0)$. We will show in a future paper that we could weaken this rule by considering only unramified extensions. 

\end{Par} 

%         The rule (R3a) is true for $\KMW$ according to the following proposition.
%         
%         \begin{Pro}[\cite{Mor12}, Lemma 2.19]
%         For any field extension $E\subset F$ and for any discrete valuation $v$ on $F$ which restricts to a valuation $w$ on $E$ with ramification index $e$. Let $\pi$ be a prime of $v$ and $\pi'$ be a prime of $w$. Write $\pi'=u\pi^e$ with $u\in \mathcal{O}_v^{\times}$. Then for each $\alpha \in \kMW_*(E)$ one has :
%         \begin{center}
%         $\partial^{\pi}_v(\alpha|_F)=e_\epsilon \langle \overline{u} \rangle ( \partial^{\pi'}_w(\alpha))|_{\kappa(v)}$.
%         \end{center}
%         \end{Pro}

\begin{Par}
The rule \ref{itm:R3b} could be understood as
\begin{center}
$\delta_v \circ \phi^{*} = \sum_w (\phi_w)^{*} \circ  \delta_w $.
\end{center}

% 
% \begin{center}
% 
% $s_v=\Theta\circ \partial_v \circ \gamma_{[-\pi]}$
% \end{center}
% where $\gamma_{[-\pi]}:\KMW(F,\VV_F)\to \KMW(F,\AAA^1+\VV_F)$ is the multiplication by $([-\pi]\otimes 1)$ and where $\Theta$ is induced by the composite of two canonical isomorphisms $-\NN_v+\AAA^1_F\simeq \AAA^1_F+(-\NN_v)\simeq 0$ of virtual vector bundles.

\end{Par}
\begin{Par}
In the rule \ref{itm:R3d}, the element $[-\pi]$ is considered to be in $\KMW(F, \AAA_F^1)$ so that the map $\partial_w \circ \gamma_{[-\pi]} \circ \phi_*$ takes value in $M(\kappa(w), -\NN_w+\AAA_{\kappa(w)}^1+ \VV_{\kappa(w)})$ (not in $M(\kappa(w), \VV_{\kappa(w)})$, as needed). In order to fix the issue, we make (as in the previous Section) a choice of isomorphism between $\AAA^1_{\kappa(w)}$ and $\NN_w$ using the uniformizer $\pi$ so that we can consider the composite of the following two\footnote{Beware of the switch isomorphism $-\NN_w+\AAA^1_{\kappa(w)}\simeq \AAA^1_{\kappa(w)}+(-\NN_w)$ which corresponds (in some sense) to multiplication by $\langle -1 \rangle$.}  isomorphisms $-\NN_w+\AAA^1_{\kappa(w)}\simeq \AAA^1_{\kappa(w)}+(-\NN_w)\simeq 0$. Thanks to the data \ref{itm:D1}, this induces an isomorphism:
\begin{center}
$\Theta_\pi:M(\kappa(w), -\NN_w+\AAA^1_{\kappa(w)}+\VV_{\kappa(w)}) \to M(\kappa(w), \VV_{\kappa(w)})$. 

\end{center}
The attentive reader might be pleased to see the rule \ref{itm:R3d} stated as:
\begin{center}
$\Theta_\pi \circ \partial_w \circ \gamma_{[-\pi]} \circ \phi_*= \overline{\phi}_*$.
\end{center}
Actually, we want more flexibility regarding the grading so that we may worry less about hidden isomorphisms. Indeed, keeping the previous notations, consider an isomorphism $\mathcal{W}\simeq \AAA^1_{\mathcal{O}_w}+\VV$ of virtual $\mathcal{O}_w$-modules. It induces (via \ref{itm:D1}) two group isomorphisms:
\begin{center}

$\Theta_{\kappa(w)}:M(\kappa(w),-\NN_v+\mathcal{W}_{\kappa(w)})\to M(\kappa(w), -\NN_v+\AAA^1_{\kappa(w)}+\VV_{\kappa(w)})$ and 
\\ 
$\Theta_{F}:M(F,\mathcal{W}_{F})\to M(F,\AAA^1_{F}+\VV_{F})$.

\end{center}
The rule \ref{itm:R3d} should state that we have the identity:
\begin{center}
$\Theta_\pi \circ \Theta_{\kappa(w)}^{-1} \circ \partial_w \circ \Theta_F \circ \gamma_{[-\pi]} \circ \phi_*= \overline{\phi}_*$.
\end{center}
 \par In this paper, we will encounter a number of similar cases where equality holds only up to a canonical isomorphism. It will always be an isomorphism $\Theta$ coming from the data \ref{itm:D1} and an isomorphism between virtual spaces.
\end{Par} 
\begin{Def}
With the previous notations, we denote by $s^\pi_w$ the morphism 
\begin{center}
$ \Theta_\pi \circ \partial_w \circ \gamma_{[-\pi]}:M(F,\VV_F)\to M(\kappa(w),\VV_{\kappa(w)})$

\end{center}
and we call it the {\em specialization map}.
\end{Def}

\begin{Par} A more rigorous way to state $\ref{itm:R3e}$ would be to write:
\begin{center}
$\partial_v \circ \gamma_{[u]\otimes 1}=-\Theta \circ  \gamma_{[\overline{u}]\otimes 1} \circ \partial_v$ and
\\ $\partial_v \circ \gamma_{\eta\otimes 1} =-\Theta' \circ \gamma_{\eta\otimes 1} \circ \partial_v$,
\end{center}
where $\Theta$ (resp. $\Theta'$) is induced by the isomorphism $\AAA^1_F+(-\NN_v)\simeq -\NN_v+\AAA^1_F$ (resp. $-\AAA^1_F+(-\NN_v)\simeq -\NN_v + (-\AAA^1_F)$) of virtual vector bundles. As in the previous remark concerning \ref{itm:R3d}, we actually want to be flexible regarding the grading and allow isomorphisms of virtual vector bundles in the axioms.
% In the rule \ref{itm:R3e}, the elements $[u]$ and $[\overline{u}]$ are considered to be in $\KMW(E, \AAA^1_E)$ and $\KMW(\kappa(v), \AAA^1_{\kappa(v)})$, respectively. Moreover, the letter $\eeta$ denotes the canonical generator of degree $-1$ as in Definition \ref{kMWdefinition}.

\end{Par}
\begin{Par} \label{RemarkOnR4a}
The rule \ref{itm:R4a} does not appear in the classical theory of \cite{Rost96}, obviously. We mainly need this rule to prove the adjunction theorem (see \ref{AdjunctionTheorem}). Keeping the notations of this rule, we note that the automorphism $\Theta:\VV_E \to \VV_E$ of virtual vector spaces induces an automorphism ${\det(\Theta):\det(\VV_E)\to \det(\VV_E)}$ (of dimension $1$ vector spaces) which corresponds to a unit $\delta_{\Theta}\in F^{\times}$; by definition $\Delta(\Theta)=\langle \delta_{\Theta} \rangle$ (recall the notation defined in \ref{NotationsKMW}).
 %to an invertible matrix also denoted by $\Theta$, so that we have %by definition $\Delta(\Theta)=\langle \det\Theta \rangle$ (recall %the notation defined in \ref{NotationsKMW}).

\end{Par} 
 
\begin{Par}
As in \cite[§1.1]{Rost96}, the rule \ref{itm:R3e} implies:
\begin{description}
\item [\namedlabel{itm:R3f}{R3f}] Let $(E,\VV_E)$ and $(E,\mathcal{W}_E)$ be two objects in $\mathfrak{F}_k$, let $v$ be a valuation on $E$ and $\pi$ be a uniformizer of $v$. For
 any $x \in \KMW (E, \mathcal{W}_E)$ and
  $\rho \in M(E,\VV_E)$, we have
\begin{center}
$\partial_v(x\cdot \rho)=\partial_v(x)\cdot s^{\pi}_v(\rho)+(-1)^n\Theta_s(s^\pi_v(x)\cdot \partial_v(\rho))+(-1)^n\Theta_{\pi}([-1]\cdot \partial_v(x) \cdot \partial_v(\rho))$,
\\ $s^\pi_v(x\cdot \rho)=s^\pi_v(x)\cdot s^\pi_v(\rho)$,
\end{center}
where $n=\rk(\mathcal{W}_E)$, $\Theta_s$ is induced by the switch isomorphism ${\mathcal{W}_E+(-\NN_v)\simeq -\NN_v+\mathcal{W}_E}$ and $\Theta_\pi$ is induced by the composite isomorphism

\begin{center}
 $(\AAA^1_{\kappa(v)}+(-\NN_v))+\mathcal{W}_{\kappa(v)}+(-\NN_v)+\VV_{\kappa(v)} \simeq 0+\mathcal{W}_{\kappa(v)}+(-\NN_v)+\VV_{\kappa(v)}\simeq -\NN_v+\mathcal{W}_{\kappa(v)}+\VV_{\kappa(v)}$.

\end{center}
\end{description} 
\end{Par}
\begin{Par}

With the previous notations, if $\pi'$ is another uniformizer of $v$, put $\pi'=u\pi$. Then (by \ref{itm:R3e} and \ref{itm:R4a})
\begin{center}
$s^{\pi'}_v(x)=s^\pi_v(x) - \Theta_{\pi'} ([\overline{u}] \cdot \partial_v(x))$,
\end{center} 
where $\Theta_{\pi'}$ is induced by the isomorphism $\AAA^1_{\kappa(v)}+(-\NN_v)\simeq 0$ defined by $\pi'$.

\end{Par}
\begin{comment}

\par From this and the rule \ref{itm:R3c} it follows that the rule \ref{itm:R3d} holds for any uniformizer $\pi$.
\end{comment}
The following theorem follows from  the results of Section \ref{MainExemple}.
\begin{The}
The functor $\KMW$ is a MW-cycle premodule.
\end{The}
\begin{proof}
The Milnor-Witt K-theory is indeed equipped with the data \ref{itm:D1} (see \ref{TwistedMWKTheory}), \ref{itm:D2} (see \ref{kMWTransfert}), \ref{itm:D3} (see \ref{TwistedMWKTheory}) and \ref{itm:D4} (see \ref{KMWresidue}). Rule \ref{itm:R1a} and rule \ref{itm:R1b} are obvious from the definitions. Rule \ref{itm:R1c} is Theorem \ref{KMWR1c}. The projection formulae \ref{itm:R2a}, \ref{itm:R2b} and \ref{itm:R2c} are straightforward computations. Rule \ref{itm:R3a} is Theorem \ref{KMWRuleR3a}. Rule \ref{itm:R3b} is Theorem \ref{KMWRuleR3b}. Rule \ref{itm:R3c}, rule \ref{itm:R3d} and rule \ref{itm:R3e} are straightforward computations (see \cite[Theorem 3.15 and Proposition 3.17]{Mor12}). 
\par We prove Rule \ref{itm:R4a}. Let $(E,\VV_E)\in \mathfrak{F}_k$ and let $\Theta$ be an endomorphism of $(E,\VV_E)$. The map $\Theta$ induces a morphism of dimension $1$ vector spaces
\begin{center}
$\det(\VV_E)\to \det(\VV_E)$\\
$l\mapsto \delta_{\Theta}\cdot l$
\end{center} 
where $\delta_{\Theta}\in E^{\times}$. We need to prove that
 \begin{center}
 
 $\Theta_*=\gamma_{\Delta(\Theta)}:\kMW_{n}(E)\otimes_{\ZZ[E^{\times}]} \det(\VV_E)\to \kMW_{n}(E)\otimes_{\ZZ[E^{\times}]} \det(\VV_E)$
 \end{center}
where $n$ is the rank of $\VV_E$ and $\Delta(\Theta)$ is the quadratic form $\langle {\delta_{\Theta}} \rangle $. Let $x\in \kMW_{n}(E)$ and $l\in \det(\VV_E)$, we have $\Theta_*(x\otimes l)=x\otimes (\delta_{\Theta}\cdot l)=(x\cdot \langle {\delta_{\Theta}} \rangle) \otimes l=\gamma_{\Delta(\Theta)}(x\otimes l)$, hence the result.
 
\end{proof}

%   \begin{Lem}
%   
%   For the validity of \ref{itm:R3d}, it suffices (under presence of the other rules of Definition \ref{defMWmodules}) to require \ref{itm:R3d} for the case $E=\kappa(v)$.
%   \end{Lem}
%   \begin{proof}
%   Verbatim. Vérifier tout de même que le pullback $\phi^*$ d'une uniformisante est encore une uniformisante...
%   \end{proof}

\begin{Def}

A pairing $M\times M'\to M''$ of MW-cycle premodules over $k$ is given by bilinear maps for each $(E,\VV_E),(E,\mathcal{W}_E)$ in $\mathfrak{F}_k$
\begin{center}

$M(E,\VV_E)\times M'(E,\mathcal{W}_E)\to M''(E,\VV_E+\mathcal{W}_E)$
\\ $(\rho, \mu)\to \rho \cdot \mu$
\end{center}
which respect the $\KMW$-module structure and which have the properties \ref{itm:P1}, \ref{itm:P2}, \ref{itm:P3} stated below.
\begin{description}
\item [\namedlabel{itm:P1}{P1}] For $x\in \KMW(E,\mathcal{W}_E), \rho \in M(E,\VV_E), \mu \in M'(E,\VV_E')$,
\item [\namedlabel{itm:P1a}{P1a}] $(x\cdot \rho) \cdot \mu= x\cdot (\rho \cdot \mu),$
\item [\namedlabel{itm:P1b}{P1b}] $ (\rho \cdot x)\cdot \mu= \rho \cdot (x \cdot \mu)$.
\end{description}
\begin{description}
\item [\namedlabel{itm:P2}{P2}] Let $\phi :E\to F$ (finite in P2b, P2c), $\lambda \in M(E,\VV_E)$, $\nu\in M(F,\Omega_{F/k}+\mathcal{W}_F)$, ${\rho \in M'(E,\VV_E')}$, $\mu\in M'(F,\Omega_{F/k}+\mathcal{W}_E')$,
\item [\namedlabel{itm:P2a}{P2a}] $\phi_*(\lambda\cdot \rho)=\phi_*(\lambda)\cdot \phi_*(\rho)$ ,
\item [\namedlabel{itm:P2b}{P2b}] $\Theta(\phi^*(\Theta'(\phi_*(\lambda)\cdot \mu)))=\lambda\cdot \phi^*(\mu)$ where $\Theta,\Theta'$ are induced by the canonical isomorphisms $\mathcal{W}_F+\Om_{F/k}\simeq \Om_{F/k}+\mathcal{W}_F$ and $\Om_{E/k}+\mathcal{W}_E \simeq \mathcal{W}_E+\Om_{E/k}$, respectively,
\item [\namedlabel{itm:P2c}{P2c}] $\phi^*(\nu\cdot \phi_*(\rho))=\phi^*(\nu)\cdot \rho$.
\end{description}
\begin{description}
\item [\namedlabel{itm:P3}{P3}] For a valuation $v$ on $E$, $x\in M(E,\VV_E)$, $\rho \in M'(E,\mathcal{W}_E)$ and a uniformizer $\pi$ of $v$, one has

\begin{center}
$\partial_v(x\cdot \rho)=\partial_v(x)\cdot s^{\pi}_v(\rho)+(-1)^n\Theta_s(s^\pi_v(x)\cdot \partial_v(\rho))+(-1)^n\Theta_\pi([-1]\cdot \partial_v(x) \cdot \partial_v(\rho))$

\end{center}
where $n$ is the rank of $\VV_E$, where $\Theta_s$ is induced by the switch isomorphism ${\mathcal{W}_E+(-\NN_v)\simeq -\NN_v+\mathcal{W}_E}$ and where $\Theta_\pi$ is induced by the composite isomorphism

\begin{center}
 $(\AAA^1_{\kappa(v)}+(-\NN_v))+\mathcal{W}_{\kappa(v)}+(-\NN_v)+\VV_{\kappa(v)} \simeq 0+\mathcal{W}_{\kappa(v)}+(-\NN_v)+\VV_{\kappa(v)}\simeq -\NN_v+\mathcal{W}_{\kappa(v)}+\VV_{\kappa(v)}$.

\end{center}
\end{description}
\par A ring structure on a MW-cycle premodule $M$ is a pairing $M\times M\to M$ which induces on $\bigoplus_{\VV_E\in \mathfrak{V}(E)}M(E,\VV_E)$ an associative and $\epsilon$-commutative ring structure (see also \ref{RingStructure}).
\end{Def}
\begin{Exe}
By definition, a Milnor-Witt cycle premodule $M$ comes equipped with a pairing $\KMW \times M \to M$. When $M=\KMW$, this defines a ring structure on $M$.

\end{Exe}

\section{Milnor-Witt Cycle Modules} \label{Modules}

\begin{Par}\label{2.0.1}

Throughout this section, $M$ denotes a Milnor-Witt cycle premodule over $k$.
\par Let $X$ be a scheme over $k$ and $\VV_X$ be a virtual bundle over $X$.
\par For $x$ in $X$ the virtual bundle $\VV_X \times_X \Spec \kappa(x)$ over $\Spec \kappa(x)$ corresponds to a virtual vector space $\VV_x$ over $\kappa(x)$ via the equivalence between vector bundles over $\Spec \kappa(x)$ and $\kappa(x)$-vector spaces.
\par Throughout this paper, we write 
\begin{center}
$M(x, \VV_X)=M(\kappa(x), \Om_{\kappa(x)/k}+\VV_x)$.
\end{center}
\par  If $X$ is irreducible, we write $\xi_X$ or $\xi$ for its generic point. 
\par If $X$ is normal, then for any $x\in X^{(1)}$ the local ring of $X$ at $x$ is a valuation ring so that \ref{itm:D4} gives us a map $\partial_x: M(\xi, \VV_X) \to M(x, \VV_X)$.

Now suppose $X$ is an arbitrary scheme over $k$ and let $x,y$ be two points in $X$. We define a map
\begin{center}
$\partial^x_y:M(x,\VV_X) \to M(y,\VV_X)$
\end{center}
as follows. Let $Z=\overline{ \{x\}}$. If $y\not \in Z$, then put $\partial^x_y=0$. If $y\in Z$, let $\tilde{Z}\to Z$ be the normalization and put
\begin{center}
$\partial^x_y=\displaystyle \sum_{z|y} \cores_{\kappa(z)/\kappa(y)}\circ \, \partial_z$
\end{center}
with $z$ running through the finitely many points of $\tilde{Z}$ lying over $y$.

\end{Par}

\begin{Def}

A Milnor-Witt cycle module $M$ over $k$ is a Milnor-Witt cycle premodule $M$ which statisfies the following conditions \ref{itm:FD} and \ref{itm:C}.
\begin{description}
\item [\namedlabel{itm:FD}{(FD)}] {\sc Finite support of divisors.} Let $X$ be a normal scheme, $\VV_X$ be a virtual vector bundle over $X$ and $\rho$ be an element of $M(\xi_X,\VV_X)$. Then $\partial_x(\rho)=0$ for all but finitely many $x\in X^{(1)}$.
\item [\namedlabel{itm:C}{(C)}] {\sc Closedness.} Let $X$ be integral and local of dimension 2 and $\VV_X$ be a virtual bundle over $X$. Then
\begin{center}
$0=\displaystyle \sum_{x\in X^{(1)}} \partial^x_{x_0} \circ \partial^{\xi}_x: M(\xi_X,\VV_X)\to M(x_0, \VV_X)$
\end{center}
where $\xi$ is the generic point and $x_0$ the closed point of $X$.
\end{description}
\end{Def}
\begin{Par}

Of course \ref{itm:C} makes sense only under presence of \ref{itm:FD} which guarantees finiteness in the sum. More generally, note that if \ref{itm:FD} holds, then for any scheme $X$, any virtual bunde $\VV_X$ over $X$, any $x\in X$ and any $\rho\in M(x, \VV_X)$ one has $\partial^x_y(\rho)=0$ for all but finitely many $y\in X$.
\end{Par}

\begin{Par} \label{DefDifferential}
If $X$ is integral and \ref{itm:FD} holds for $X$, we put
\begin{center}
$d=(\partial^\xi_x)_{x\in X^{(1)}}:M(\xi, \VV_X) \to \displaystyle \bigoplus_{x\in X^{(1)}} M(x, \VV_X)$.
\end{center}

\end{Par}

\begin{Def} \label{DefMWmorphisms}
\begin{description}
\item 
A morphism $\omega:M\to M'$ of Milnor-Witt cycle modules over $k$ is a natural 
transformation which commutes\footnote{In particular, it commutes with multiplication by $\eeta$ (up to canonical isomorphisms). This fact will be important for Theorem \ref{AdjunctionTheorem}.} with the data \ref{itm:D1},\dots, \ref{itm:D4} (in other words, it is a left map in the sense of \cite[Definition 40]{Yetter03}).

\end{description}

\end{Def}
\begin{Rem}\label{MWCyleAbelian}

We denote by $\CatMW_k$ the category of Milnor-Witt cycle modules (where arrows are given by morphisms of MW-cycle modules). This is an abelian category.
\end{Rem}
\begin{Exe}\label{ExampleColimit}

Let $M$ be a Milnor-Witt cycle module. For any integer $n$, we define a Milnor-Witt cycle module $M\{n\}$ by 
\begin{center}

$M\{n\}(E,\VV_E)=M(E,n\cdot \AAA^1_E+\VV_E)$
\end{center}
for any (finitely generated) field $E/k$ and any virtual bundle $\VV_E$ over $E$.
\par Now, we define a Milnor-Witt cycle module $M[\eeta^{-1}]$ as follows. Let $I$ be the category of finite ordinal numbers (objects are natural numbers $n$ and arrows are given by the relation $n\leq m$). We consider the functor $F_M:I\to \CatMW_k$ that takes a natural number $n$ to $M\{-n\}$ and an arrow $n\leq m$ to the multiplication by $\eeta^{m-n}$. By definition, the Milnor-Witt cycle module $M[\eeta^{-1}]$ is the colimit of the diagram $F_M$. In particular, for $M=\KMW$ and for any field $E$, we obtain the following group isomorphism which respects the multiplication in some obvious sense:
\begin{center}

$\bigoplus_{n\in \ZZ} \KMW [\eeta^{-1}](E,n\cdot \AAA^1_E)\simeq \kMW_*(E)[\eeta^{-1}]$.
\end{center}  
The ring $\kMW_*(E)[\eeta^{-1}]$ is known to be isomorphic to $ W(E)[\eeta^{\pm 1}]$.
\end{Exe}

\begin{Par}

In the following, let $F$ be a field over $k$, $\VV_F$ be a virtual bundle over $\Spec F$ and ${\AAA_F^1=\Spec F[t]}$ be the affine line over $\Spec F$ with function field $F(t)$. 
\end{Par}

\begin{Pro} \label{Prop2.2}

Let $M$ be a Milnor-Witt cycle module over $k$. With the previous notations, the following properties hold.
\begin{description}
\item [\namedlabel{itm:(H)}{(H)}]  {\sc  Homotopy
 property for $\AAA^1$}. We have a short exact sequence
%  \footnote{Recall \ref{KahlerDiffFractionField}}

\begin{center}
$\xymatrix@C=10pt@R=20pt{
0 \ar[r] &  M(F,\AAA^1_F+\Om_{F/k}+\VV_F) \ar[r]^-{\res}   & M(F(t),\Om_{F(t)/k}+ \VV_{F(t)}) \ar[r]^-d  &  \bigoplus_{x\in {(\AAA_F^1)}^{(1)}} M(\kappa(x), \Om_{\kappa(x)/k}+\VV_{\kappa(x)}) \ar[r] & 0
}
$
\end{center}
where the map $d$ is defined in \ref{DefDifferential}.

\begin{comment}

$
\xymatrix{
0 \ar[r] &  M(F,\AAA^1_F+\Om_{F/k}+\VV_F) \ar[r]^{\res}   & M(F(t),\Om_{F(t)/k}+ \VV_{F(t)})  \\
  & \, \, \,\, \, \,\, \, \,\, \, \,\, \, \, \, \, \, \ar[r]^-d
 & \bigoplus_{x\in {(\AAA_F^1)}^{(1)}} M(\kappa(x),\Om_{\kappa(x)/k}+ \VV_{\kappa(x)}) \ar[r] & 0
}
$
\end{comment}
% %  \resizebox{17cm}{!}{
%
% }
\item [\namedlabel{itm:(RC)}{(RC)}] {\sc Reciprocity for curves}. Let $X$ be a proper curve over $F$ and $\VV_F$ a virtual bundle\footnote{We also write $\VV_F$ for the corresponding virtual $F$-vector space.} over $\Spec F$. Then
\begin{center}

$\xymatrix{
M(\xi_X, \VV_X) \ar[r]^-d & \displaystyle \bigoplus_{x\in X^{(1)}} M(x, \VV_X) \ar[r]^c & M(F,  \Om_{F/k}+\VV_F)
}$

\end{center}
is a complex, that is $c\circ d=0$ (where $c={\sum_x \cores_{\kappa(x)/F}}$).
\end{description}
\end{Pro}

\begin{Par}
Axiom \ref{itm:FD} enables one to write down the differential $d$ of the soon-to-be-defined complex $C_*(X,M,\VV_X)$, axiom \ref{itm:C} guarantees that $d\circ d = 0$, property \ref{itm:(H)} yields the homotopy invariance of the Chow groups $A_*(X,M,\VV_X)$ and finally \ref{itm:(RC)} is needed to establish proper pushforward.

\end{Par}

\begin{Par}

For an integral scheme $X$ and $\VV_X$ a virtual bundle over $X$, we put
\begin{center}
$A^0(X,M,\VV_X)=\ker d= \displaystyle \bigcap_{x\in X^{(1)}} \ker \partial^\xi_x \subset M(\xi_X,\VV_X)$.
\end{center}
\end{Par}

\begin{The} \label{Thm2.3}

Let $M$ be a Milnor-Witt cycle premodule over the perfect field $k$. Then $M$ is a cycle module if and only if the following properties \ref{itm:(FDL)} and \ref{itm:(WR)} hold for all fields $F$ over $k$ and all virtual $F$-vector space $\VV_F$.
\begin{description}
\item  [\namedlabel{itm:(FDL)}{(FDL)}] {\sc Finite support of divisors on the line}. Let $\rho \in M(F(t),\VV_{F(t)})$. Then $\partial_v(\rho)=0$ for all but finitely many valuations $v$ of $F(t)$ over $F$.
\item [\namedlabel{itm:(WR)}{(WR)}] {\sc Weak Reciprocity}. Let $\partial_\infty$ be the residue map for the valuation of $F(t)/F$ at infinity. Then
\begin{center}
$\partial_\infty(A^0(\AAA_F^1, M , \VV_{\AAA_F^1}))=0$.
\end{center}
\end{description}
\end{The}

%    \begin{Pro}

%    
%    Further properties of Milnor-Witt cycle modules are the following.
%    
%    \begin{itemize}
%    \item {\sc (Co) Continuity}. Let $X$ be smooth and local with a virtual bundle $\VV_X\to X$ and let $Y\to X$ be the blow up in the unique closed point $x_0$. Then
%    \begin{center}
%    $A^0(X,M,\VV_X) \subset A^0(Y,M,\VV_Y)$.
%    \end{center}
%    \item {\sc (E) Evaluation}. In the situation of (Co) there is a unique morphism
%    \begin{center}
%    $\operatorname{ev}:A^0(X,M,\VV_M)\to M(x_0,\VV_X)$
%    \end{center}
%    ("evaluation at $x_0$) such that 
%    \begin{center}
%    
%    $\res_{\kappa(v)/\kappa(x_0)}\circ \operatorname{ev}=s^\pi_v|_{A^0(X,M,\VV_X)}$
%    \end{center}
%    for any prime $\pi$ of $v$.
%    \end{itemize}
%    \end{Pro}
%            \begin{proof}
%            
%            This will follow from the construction of the pullback map $f^*:A^0(X,M,\VV_M)\to A^0(Y,M,\VV_Y - \TT_{Y/X})$ for morphism $f:Y\to X$ in ??. Namely, the inclusion of (Co) is given by $f^*$ with $f:Y\to X$ the blow up. Moreover in (E), one has $\operatorname{ev}=f^*$ with $f:\Spec \kappa(x_0) \to X	$ the inclusion.
%            \end{proof}

%            In the following proofs, we use the notations $\AAA^1_F=\Spec F[u]$, $\AAA_F^2=\Spec F[s,t]$ and $Z=\AAA^2_{(\langle s,t \rangle>)}$, the localization of $\AAA^2_F$ at $0$. Moreover $y,z\in Z^{(1)} \subset (\AAA^2_F)^{(1)}$ denote the points with parameters $s,t$, respectively. We proceed in several steps.

The proofs of Proposition \ref{Prop2.2} and Theorem \ref{Thm2.3} are almost the same as Rost's (see \cite[Theorem 2.3]{Rost96}). Indeed, we can check that all involved twists match (up to canonical isomorphisms) and, of course, we use the identity ${[ab]=[a]+[b]+\eeta[a][b]}$ every time Rost use ${[ab]=[a]+[b]}$ (this causes no harm to the proof since $\eeta$ commutes with our data).

\begin{The}\label{KMWisModule}

The Milnor-Witt K-theory $\KMW$ is a MW-cycle module.
\end{The}
\begin{proof}
We already know that it is a Milnor-Witt cycle premodule. For the two remaining axioms, it suffices to prove \ref{itm:FD} and \ref{itm:(H)} which are true (see Section §\ref{KMWFD} and Theorem \ref{KMWhmtpyInv}; see also \cite[Theorem 3.24]{Mor12}). 
\end{proof}

\section{The Five Basic Maps} \label{FiveBasic}

The purpose of this section is to introduce the cycle complexes and each operation on them needed further on. Note that the five basic maps defined below are analogous to those of Rost (see \cite[§3]{Rost96}); they are the basic foundations for the construction of more refined maps such as Gysin morphisms (see Section \ref{GysinMorphisms}).
\begin{Par}

Let $M$ and $N$ be two Milnor-Witt cycle modules over $k$, let $X$ and $Y$ be two schemes, let $\VV_X$ and $\VV_Y$ be two virtual bundles over $X$ and $Y$ respectively, and let $U\subset X$ and $V\subset Y$ be subsets. Given a morphism
\begin{center}
$\alpha:\displaystyle \bigoplus_{x\in U} M(x,\VV_X) \to \displaystyle \bigoplus_{y\in V} M(y, \VV_Y)$,
\end{center}
we write $\alpha^x_y:M(x,\VV_X) \to M(y, \VV_Y)$ for the components of $\alpha$.
\end{Par}

%    \begin{Par}{\sc Change of coefficients}
%    Let $\omega:M\to N$ be a morphism between two Milnor-Witt cycle modules, let $X$ be a scheme over $k$ and $U\subset X$ a subset. We put
%    \begin{center}
%    
%    $\omega_{\#}:\bigoplus_{x\in U} M(x,\VV_X) \to \bigoplus_{x\in U} N(x,\VV_X)$
%    \end{center}
%    where $(\omega_{\#})^x_x=\omega_{\kappa(x)}$ and $(\omega_{\#})^x_y=0$ for $x\neq y$.

%    \end{Par}
\begin{Par}
{\sc Milnor-Witt cycle complexes}.
Let $M$ be a Milnor-Witt cycle module, let $X$ be a scheme, $\VV_X$ be a virtual bundle over $X$ and $p$ be an integer. Put $X_{(p)}$ the set of $p$-dimensional points of $X$. Define
\begin{center}

$C_p(X,M,\VV_X)=\displaystyle \bigoplus_{x\in X_{(p)}}M(x,\VV_X)$
\end{center}
and
\begin{center}
$d:C_p(X,M,\VV_X)\to C_{p-1}(X,M,\VV_X)$
\end{center}
where $d^x_y=\partial^x_y$ as in \ref{2.0.1}. This definition makes sense by axiom \ref{itm:FD}.
\end{Par}

\begin{Pro} \label{dod=0}
With the previous notations, we have $d\circ d=0$.
\end{Pro}
\begin{proof}
Same as in \cite[§3.3]{Rost96}. Axiom \ref{itm:C} is needed.
\end{proof}

\begin{Def} \label{DefinitionComplex}
The complex $(C_p(X,M,\VV_X),d)_{p\geq 0}$ is called the {\em Milnor-Witt complex of cycles on $X$ with coefficients in $M$}.
\end{Def}

\begin{Par}{\sc Pushforward}
Let $f:X\to Y$ be a $k$-morphism of schemes, let $\VV_Y$ be a virtual bundle over $Y$ and denote by $\VV_X$ its pullback along $f$. Define
\begin{center}

$f_*:C_p(X,M,\VV_X)\to C_p(Y,M, \VV_Y)$
\end{center}
as follows. If $y=f(x)$ and if $\kappa(x)$ is finite over $\kappa(y)$, then $(f_*)^x_y=\cores_{\kappa(x)/\kappa(y)}$. Otherwise, $(f_*)^x_y=0$.
\end{Par}

\begin{Par}{\sc Pullback} \label{pullbackBasicMap}
Let $f:X\to Y$ be an {\em essentially smooth} morphism of schemes. Let $\VV_Y$ a virtual bundle over $Y$ and $\VV_X$ be its pullback along $f$. Suppose $X$ connected and denote by $s$ the relative dimension of $f$. Define
\begin{center}
$f^*:C_p(Y,M,\VV_Y) \to C_{p+s}(X,M,- \TT_{X/Y}+\VV_X)$
\end{center}
as follows. If $f(x)=y$, then $(f^*)^y_x=\Theta\circ \res_{\kappa(x)/\kappa(y)}$, where $\Theta$ is the canonical isomorphism induced by $\TT_{\Spec\kappa(x)/\Spec\kappa(y)}\simeq \TT_{X/Y} \times_X \Spec \kappa(x)$. Otherwise, $(f^*)^y_x=0$. If $X$ is not connected, take the sum over each connected component.

\end{Par}

%        \begin{proof}
%        
%        Let $x\in X_{(p+s)}$, $f(x)=y\in Y_{(p)}$.
%        
%        \begin{center}
%        
%        $0\to \Om_{\kappa(y))/k}\otimes_{\kappa(y)} \kappa(x) \to \Om_{\kappa(x)/k} \to
%        \Om_{\kappa(x)/\kappa(y)} \to 0$
%        \end{center}
%        
%        $T=\overline{ \{y\}}_{red}$ $y\in T^{(0)}$ $x\in (X\times_Y T)^{(0)}$
%        \begin{center}
%        
%        $\xymatrix{
%        \kappa(x) \ar[r]^{\text{pro-ouverte}} \ar[d]_{\text{essent. lisse ; formell. lisse}} & X\times_Y T \ar[r] \ar[d]_{\text{lisse}} & X \ar[d]_{\text{lisse}} \\
%        \kappa(y) \ar[r]_{\text{pro-ouverte}} & T \ar[r] & Y%        
%        }$
%        
%        \end{center}
%        $Z=X\times_Y T $. We have
%        \begin{center}
%        
%        $\Om_{Z/T} \simeq \Om_{X/Y} \times_Y T$
%        \end{center}
%        \begin{center}
%        
%        $\Om_{Z/T} \times_{\kappa(y)} \kappa(x) \simeq %        {\kappa(x)/\kappa(y)}$.
%        \end{center}
%        
%        
%        

%        \end{proof}
\begin{Rem}
The fact that the morphism $f$ is (essentially) smooth implies that there are no multiplicities to consider. We do not consider the case of flat morphisms in this paper (this can be done after studying the multiplicities mentioned in Claim \ref{StrongR1c}). However, we define Gysin morphisms for lci projective morphisms in Section \ref{GysinMorphisms}.
\end{Rem}

\begin{Par}{\sc Multiplication with units}
Let $a_1,\dots, a_n$ be global units in $\mathcal{O}_X^*$, let $\VV_X$ be a virtual bundle. Define
\begin{center}
$[a_1,\dots, a_n]:C_p(X,M,\VV_X) \to C_p(X,M,n\cdot \AAA^1_X+\VV_X)$
\end{center}
as follows. Let $x$ be in $X_{(p)}$ and $\rho\in M(\kappa(x),\Om_{ \kappa(v)/k}+\VV_x)$. We consider\footnote{Instead of multiplying by $\langle -1 \rangle$ and using \ref{itm:R4a}, we could make a canonical choice of isomorphism with determinant $(-1)$ and use \ref{itm:D1} (by \ref{itm:R4a}, any choice will work).} $\langle-1 \rangle^{np}[a_1(x),\dots, a_n(x)]$ as an element of ${\KMW (\kappa(x),n\cdot \AAA^1_{\kappa(x)})}$.
If $x=y$, then put $[a_1,\dots , a_n]^x_y(\rho)=\Theta(\langle-1 \rangle^{np}[a_1(x),\dots , a_n(x)]\cdot \rho) $ where $\Theta$ is induced by the canonical isomorphism $n\cdot \AAA^1_{\kappa(x)}+\Om_{\kappa(x)/k}\simeq \Om_{\kappa(x)/k}+n\cdot  \AAA^1_{\kappa(x)}$. Otherwise, put $[a_1,\dots , a_n]^x_y(\rho)=0$.

\end{Par}

\begin{Par}{\sc Multiplication with $\eeta$}
Define
\begin{center}

$\eeta:C_p(X,M,\VV_X)\to C_p(X,M,-\AAA^1_X+\VV_X)$
\end{center}
as follows. If $x=y$, 
then $\eeta^x_y(\rho)
=\gamma_{\eeta}(\rho)$. 
Otherwise, $\eeta^x_y(\rho)=0$.

\end{Par}

\begin{Par}{\sc Boundary maps} \label{BoundaryMaps}
Let $X$ be a scheme of finite type over $k$ with a virtual bundle $\VV_X$, let $i:Z\to X$ be a closed immersion and let $j:U=X\setminus Z \to X$ be the inclusion of the open complement. We will refer to $(Z,i,X,j,U)$ as a boundary triple and define
\begin{center}

$\partial=\partial^U_Z:C_p(U,M,\VV_U) \to C_{p-1}(Z,M,\VV_Z)$
\end{center}
by taking $\partial^x_y$ to be as the definition in \ref{2.0.1} with respect to $X$. The map $\partial^U_Z$ is called the boundary map associated to the boundary triple, or just the boundary map for the closed immersion $i:Z\to X$.

\end{Par}

\begin{Par}{\sc Generalized correspondences} \label{GeneralizedCorr}
We will use the notation
\begin{center}
$\alpha  : [X,\VV_X] \bullet\!\!\! \to [Y,\VV_T]$

\end{center}
or simply 
\begin{center}
$\alpha  : X \bullet\!\!\! \to Y$

\end{center}
to denote maps of complexes which are sums of composites of the five basics maps $f_*$, $ g^*$, $ [a]$, $\eeta$, $ \partial$ for schemes over $k$.
Unlike Rost in \cite[§3]{Rost96}, we look at these morphisms up to quasi-isomorphisms so that a morphism $\alpha  : X \bullet\!\!\! \to Y$ may be a weak inverse of a well-defined morphism of complexes.

\end{Par}

\section{Compatibilities} \label{Compatibilities}
In this section we establish the basic compatibilities for the maps considered in the last section. Fix $M$ a Milnor-Witt cycle module.

\begin{Pro} \label{Prop4.1}

\begin{enumerate}
\item Let $f:X\to Y$ and $f':Y\to Z$ be two morphisms of schemes. Then
\begin{center}
$(f'\circ f)_*=f'_* \circ f_*$.

\end{center} 
\item Let $g:Y\to X$ and $g':Z\to Y$ be two essentially smooth morphisms. Then (up to the canonical isomorphism given by $\TT_{Z/X}\simeq \TT_{Z/Y}+(g')^*\TT_{Y/Z}$):
\begin{center}

$(g\circ g')^*=g'^*\circ g^*$.
\end{center}
\item Consider a pullback diagram
\begin{center}

$\xymatrix{
U \ar[r]^{g'} \ar[d]_{f'} & Z \ar[d]^f \\
Y \ar[r]_{g} & X
}$
\end{center}
with $f,f',g,g'$ as previously. Then
\begin{center}

$g^*\circ f_* = \Theta \circ f'_* \circ g'^*$
\end{center}
where $\Theta$ is the canonical isomorphism induced by $\TT_{U/Z} \simeq \TT_{Y/X}\times_Y U$.
\end{enumerate}
\end{Pro}
\begin{proof}
\begin{enumerate}
\item This is clear from the definition and by \ref{itm:R1b}.
% \footnote{Not by \ref{itm:R1a} as claimed in \cite{Rost96}.}

\item The claim is trivial by \ref{itm:R1a} (again, there are no multiplicities).
\item This reduces to the rule \ref{itm:R1c} (see \cite[Proposition 4.1]{Rost96}).
\end{enumerate}
\end{proof}

\begin{Pro} \label{Lem4.2}

Let $f:Y\to X$ be a morphism of schemes. If $a$ is a unit on $X$, then
\begin{center}

$f_*\circ [\tilde{f}^*(a)]=[a]\circ f_*$
\end{center} where $\tilde{f}^*:\mathcal{O}^*_X\to \mathcal{O}^*_Y$ is induced by $f$.
%                  \item Let $f$ be finite and  smooth, and let $a'$ be a unit in $Y$. Then
%                 \begin{center}
%                 
%                 $f_*\circ [a] \circ f^*=\tilde{f}_*([a'])]$.\footnote{Attention : bien définir le second membre car il y a un twist. Ca devrait être une sorte d'application du $\phi^*:K^{MW}_1\to K^{MW}_1$ d'où le twist.}  
%                 \end{center}
%                 as morphism
%                 \begin{center}

%                 $C_p(X,M,\VV_X)\to C_{p+s}(X,M,\VV_X+\mathcal{W}_X+\TT_{X/k})$
%                 \end{center}
%                 and
%                 where $\tilde{f}_*:\mathcal{O}_Y \to \mathcal{O}_X$
\end{Pro}
\begin{proof}

This comes from \ref{itm:R2b}.
%                 \item We may assume $X=\Spec F$ with $F$ a field. Then for $y\in Y$. We have
%                 \begin{center}
%                 
%                 $f_*\circ [a] \circ f^*=\sum_y \cores_{\kappa(y)/F}([a(y)]\cdot \res_{\kappa(y)/F})$
%                 \end{center}
%                 and the claim follows from R2c.
\end{proof}

\begin{Pro} \label{Lem4.3}

Let $a$ be a unit on a scheme $X$.
\begin{enumerate}
\item Let $g:Y\to X$ be an essentially smooth morphism. One has 
\begin{center}

$g^*\circ [a]=[\tilde{g}^*(a)]\circ g^*$.
\end{center}
\item Let $(Z,i,X,j,U)$ be a boundary triple. One has
\begin{center}

$\partial^U_Z \circ [\tilde{j}^*(a)]=\epsilon  [\tilde{i}^*(a)]\circ \partial^U_Z$.
\end{center}
Moreover,
\begin{center}
$\partial^U_Z \circ \eeta=\eeta \circ \partial^U_Z$.

\end{center}

\end{enumerate}
\end{Pro}
\begin{proof}
The first result comes from \ref{itm:R2a}, the second from \ref{itm:R2b} and \ref{itm:R3e}.
\end{proof}

\begin{Pro} \label{Prop4.4}

Let $h:X\to X'$ be a morphism of schemes. Let $Z'\hookrightarrow X'$ be a closed immersion. Consider the induced diagram given by $U'=X'\setminus Z'$ and pullback:

\begin{center}
$\xymatrix{
Z \ar@{^{(}->}[r] \ar[d]^{f} & X \ar[d]^h & U \ar@{_{(}->}[l] \ar[d]^{g} \\
Z' \ar@{^{(}->}[r] & X'  & U'. \ar@{_{(}->}[l]
}$
\end{center}
\begin{enumerate}
\item If $h$ is proper, then
\begin{center}

$f_*\circ \partial^U_Z = \partial^{U'}_{Z'} \circ g_*.$
\end{center}
\item If $h$ is essentially smooth, then
\begin{center}

$f^*\circ \partial^{U'}_{Z'} = \partial^U_Z \circ g^*.$
\end{center}
\end{enumerate}

\end{Pro}
\begin{proof}
This will follow from Proposition \ref{Prop4.6}.
\end{proof}

\begin{Lem}

\label{Lem4.5}
Let $g:Y\to X$ be a smooth morphism of schemes of finite type over a field of constant fiber dimension $1$, let $\sigma:X\to Y$ be a section of $g$ and let $t\in \mathcal{O}_Y$ be a global parameter defining the subscheme $\sigma(X)$. Moreover, let $\tilde{g}:U \to X$ be the restriction of $g$ where $U=Y\setminus \sigma(X)$ and let $\partial$ be the boundary map associated to $\sigma$. Then
\begin{center}
 $\Theta \circ \partial \circ {[t]} \circ \tilde{g}^*=(\id_X)_*$ and $\partial \circ \tilde{g}^*=0$,
\end{center}
with $\Theta$ the canonical isomorphism given by $\TT_{U/X}\simeq \AAA^1_U$.

\end{Lem}
 \begin{proof} %In short, one reduces to $X=\Spec E$ where $E$ is a field and applies \ref{itm:R3c} and \ref{itm:R3d}. We give more details for the second equation.
 In order to simplify the notations, we identify $X$ with $\sigma(X)$ through $\sigma$ and we forget the canonical isomorphism $\Theta$. By definition, the map $\partial \circ [t] \circ \tilde{g}^*$ is 
\begin{center}
$\displaystyle \sum_{x\in X} \sum_{y\in g^{-1}(x)\setminus \{x\}} \partial^y_x\circ [t]\circ \res_{\kappa(y)/\kappa(x)} :
\bigoplus_{x\in X_{}}M(x,\VV_X) \to\bigoplus_{x\in X_{}}M(x,\VV_X).
$
\end{center}
Fix $x\in X$ and let $y\in g^{-1}(x)\setminus \{x\}$. If $y=\xi_x$ is the generic point of $g^{-1}(x)$, then we have 
\begin{center}
${\partial^{\xi_x}_x\circ [t]\circ \res_{\kappa(\xi_x)/\kappa(x)}=\res_{\kappa(x)/\kappa(x)}=\Id}$
\end{center} according to \ref{itm:R3d}. Otherwise 
\begin{center}
${\partial^y_x\circ [t]\circ \res_{\kappa(y)/\kappa(x)}=\epsilon\cdot [t]\cdot \partial^y_x\circ \res_{\kappa(y)/\kappa(x)}=0}$
\end{center} according to \ref{itm:R3e} and \ref{itm:R3c}. The second equality is proved in a similar fashion.

% The remaining sum $\sum_{y\in g^{-1}(x)\setminus \{x,\xi_x\}} \partial^y_x\circ [t]\circ \res_{\kappa(y)/\kappa(x)}$ is zero according to Proposition \ref{Prop2.2} (Reciprocity for curves). We recall that this last result comes from the fact that there is a finite map $g^{-1}(x)\to \PP^1_{\Spec \kappa(x)}$. Hence, (using \ref{itm:R3b}), we can reduce to the case $g^{-1}(x)=\PP^1_{\Spec \kappa(x)}$ and conclude thanks to the homotopy invariance property $\ref{itm:(H)}$.

 \end{proof}

\begin{Pro}\label{Prop4.6}
\begin{enumerate}
\item Let $f:X\to Y$ be a proper morphism of schemes. Then
\begin{center}

$d_Y\circ f_*= f_*\circ d_X$.
\end{center}
\item Let $g:Y\to X$ be an essentially smooth morphism. Then
\begin{center}

$g^*\circ d_X=d_Y \circ g^*$.
\end{center}
\item Let $a$ be a unit on $X$. Then
\begin{center}

$d_X \circ [a]=\epsilon[a]\circ d_X$.
\end{center}
Moreover,
\begin{center}

$d_X \circ \eeta=\eeta \circ d_X$.
\end{center}
\item Let $(Z,i,X,j,U)$ be a boundary triple. Then
\begin{center}

$d_Z\circ \partial^U_Z=-\partial^U_Z\circ d_U$.
\end{center}
\end{enumerate}

\end{Pro}

\begin{proof} Same as \cite[Proposition 4.6]{Rost96}. The first assertion comes from Proposition \ref{Prop4.1}.1 and \ref{itm:R3b}, the second  from \ref{itm:R3c}, Proposition \ref{Prop4.1}.3, Proposition \ref{Prop4.6}.1 and \ref{itm:R3a} (note that the proof is actually much easier in our case since there are no multiplicities to consider). The third assertion follows from the definitions and Proposition \ref{Lem4.3}.2, the fourth from the fact that $d\circ d=0$.
%     \begin{enumerate}
%     \item This comes from \ref{Prop4.1} 1) and \ref{itm:R3b} (see \cite{Rost96}, Proposition 4.6).
%     \item  This comes from \ref{itm:R3c}, \ref{Prop4.1}.3 and \ref{Prop4.6}.1 and \ref{itm:R3a} (see \cite{Rost96}, Proposition 4.6 ; note that the proof is much more simpler in our case because there are no multiplicities to consider).
%     Let $\delta=d_Y\circ g^* - g^*\circ d_X$. We have to show $\delta^x_y=0$ for $x\in X_{(p)}, u\in Y_{(p+s-1)}$. Put $z=g(y)$. If $z\not \in \overline{ \{x\}}$, the claim is obvious. If $z=x$, then for $u\in Y_x$ all valuation on $\kappa(u)$ with center $y$ are trivial on $\kappa(x)$ ; the claim follows from rule \ref{itm:R3c}. We are now reduced to the case $z\in \overline{ \{x\}}, z\neq x$. Then $z\in X_{(p-1)}$ since $\dim_X(z) \geq \dim_Y(y)-s=p-1$. We may assume $X=\overline{ \{x\}}$. Moreover by \ref{Prop4.1}.3) and \ref{Prop4.6}.1) we may assume additionnaly that $X$ is normal. Let $U=\{ u \in Y_x^{(0)} \, | y\in \overline{\{u\}}\}$. Then
%     \begin{center}
%     
%     $\delta^x_y=\displaystyle \sum_{u\in U} \partial^u_y\circ (g^*)^x_u - (g^*)^z_y \circ \partial^x_y$.
%     \end{center}
%     We may replace $X$ and $Y$ by its localization in $z$ and $y$, respectively. Then $X=\Spec R$ with $R$ a valuation ring, $Y=\Spec S$ with $S$ local of dimension $1$ and $U=Y_{(1)}$ is a singleton. Hence $\delta^x_y=0$ by \ref{itm:R3a}.
%     \end{enumerate}
\end{proof}
\begin{Rem}

Note the quadratic form $\epsilon$ playing the role of $(-1)$ in the third formula but not in the fourth.
\end{Rem}

\section{Milnor-Witt Cycle Complexes and Chow-Witt Groups}\label{MWComplexe}
\begin{Par}
{\sc Milnor-Witt Cycle Complexes}
Let $M$ be a Milnor-Witt cycle module, let $X$ be a scheme with a virtual bundle $\VV_X$ and $p$ an integer. We have defined (see Definition \ref{DefinitionComplex})
\begin{center}

$C_p(X,M,\VV_X)=\displaystyle \bigoplus_{x\in X_{(p)}}M(x,\VV_X)$
\end{center}
with differential
\begin{center}

$d=d_X:C_p(X,M,\VV_X)\to C_{p-1}(X,M,\VV_X)$.
\end{center}
In the same way, we can define 
\begin{center}

$C^p(X,M,\VV_X)=\displaystyle \bigoplus_{x\in X^{(p)}}M(x,\VV_X)$
\end{center}
with differential
\begin{center}

$d=d_X:C^p(X,M,\VV_X)\to C^{p+1}(X,M,\VV_X)$
\end{center}
 and show that this gives us another complex.

\end{Par}
\begin{Def}

The {\em Chow-Witt group of $p$-dimensional cycles with coefficients in $M$} is defined as the $p$-th homology group of the complex $C_*(X,M,\VV_X)$ and denoted by $A_p(X,M,\VV_X)$. Similarly, we define a cohomology group $A^p(X,M,\VV_X)$ with $C^*(X,M,\VV_X)$.
\end{Def}
\begin{Par}

According to the previous section (see Proposition \ref{Prop4.6}), the morphisms $f_*$ for $f$ proper, $g^*$ for $g$ essentially smooth, multiplication by $[a_1,\dots , a_n]$ or $\eeta$, $\partial^U_Y$ (anti)commute with the differentials.

\end{Par}
\begin{Par}\label{LocalizationSequence}

Let $(Z,i,X,j,U)$ be a boundary triple. We can split the complex $C_*(X,M,\VV_X)$ as
\begin{center}

$C_*(X,M,\VV_X)=C_*(Z,M,\VV_Z)\oplus C_*(U,M,\VV_U)$
\end{center}so that there is a long exact sequence
\begin{center}

$\xymatrix{
\dots \ar[r]^-\partial &  A_p(Z,M,\VV_Z) \ar[r]^{i_*} & A_p(X,M,\VV_X) \ar[r]^{j^*} 
& A_p(U,M,\VV_U) \ar[r]^{\partial} & A_{p-1}(Z,M,\VV_Z) \ar[r]^-{i_*} & \dots.
}$
\end{center}
\end{Par}

\begin{Par}
{\sc Classical Chow-Witt Groups.} \label{ClassicalChowWittGroups}
Recall some definitions about the classical theory of Chow-Witt groups (see \cite{Fasel13}). We note that the results below are true in any characteristic.
\par Let $X$ be a smooth scheme and let $p,r\in \ZZ$. For $x\in X^{(p)}$, $\NN_x$ is a virtual $\kappa(x)$-vector space of dimension $p$. Denote $\Lambda_x=\bigwedge^p \NN_x$ its determinant and $\Lambda_x^*$ its dual.
\par The Rost-Schmid complex in Milnor-Witt K-theory $C(X, \mathbf{K}_r^{\text{MW}})$ is the complex
\begin{center}

$\dots \to \displaystyle \bigoplus_{x\in X^{(p)}} \kMW_{r-p}(\kappa(x),p,\Lambda^*_x)
\to  \displaystyle \bigoplus_{x\in X^{(p+1)}} \kMW_{r-p-1}(\kappa(x),p+1, \Lambda^*_x) \to \dots $
\end{center}
where $ \kMW_{r-p}(\kappa(x),p,\Lambda^*_x)=\KMW(\kappa(x), \Lambda_x^*+(r-p-1)\cdot \AAA^1_{\kappa(x)})$ (see also \cite[Definition 5.7]{Mor12}).
Denote by $\CHt^p(X)_r$ its $p$-th cohomology group. By definition, the classical Chow-Witt group is $\CHt^p(X)_p$ and is simply denoted by $\CHt^p(X)$. It is related with our previous constructions as follows.

\end{Par}

\begin{Pro} Let $X$ be a smooth scheme and let $p,r\in \ZZ$. Then we have a canonical isomorphism
\begin{center}

$A^p(X, \KMW, - \TT_{X/k}+r\cdot \AAA^1_X ) \simeq \CHt^p(X)_r$.
\end{center}
In particular with $r=p$,
\begin{center}

$A^p(X, \KMW, - \TT_{X/k}+p\cdot \AAA^1_X) \simeq \CHt^p(X)$
\end{center}

\end{Pro}
\begin{proof}

Let $x$ be in $X^{(p)}$. We have the following canonical isomorphisms
\begin{center}

\begin{tabular}{lll}
$\kMW_{r-p}(\kappa(x),p, \Lambda_x^*)$ & $\simeq $ & $\KMW (\kappa(x), \Lambda_x^*+(r-p-1)\cdot \AAA_{\kappa(x)}),$  \\
& $\simeq $ & $\KMW (\kappa(x),  - \Lambda_x+2\cdot \AAA_{\kappa(x)} + (r-p-1) \cdot \AAA_{\kappa(x)}),$ \\
   &  $\simeq $   &  $\KMW(\kappa(x), -( \Lambda_x +(p-1)\cdot \AAA_{\kappa(x)})+r\cdot \AAA_{\kappa(x)}),$ \\
     &  $\simeq $   & $\KMW(\kappa(x), - \NN_x + r\cdot \AAA_{\kappa(x)}),$ \\
     &   $\simeq $ & $\KMW(\kappa(x), \Om_{ \kappa(x)/k} - \TT_{X/k}\times_X \kappa(x)+r\cdot \AAA_{\kappa(x)}),$ \\
     &    $\simeq $ & $ \KMW(x, - \TT_{X/k}+r\cdot \AAA_X),$

\end{tabular}

\end{center}

and so
\begin{center}

$\displaystyle \bigoplus_{x\in X^{(p)}} \kMW_{r-p} (\kappa(x),p,\Lambda^*_x) \simeq  C^p(X, \KMW,  - \TT_{X/k}+r\cdot \AAA_X)$.
\end{center}
By definition, the differentials agree, hence the result.
\end{proof}

\section{Acyclicity for Smooth Local Rings} \label{Acyclicity}
Fix $M$ a MW-cycle module over $k$. We follow \cite[§6]{Rost96}. We prove that the cohomology of the complex associated to $M$ can be computed using the Zariski sheaf given by 
\begin{center}

$U\mapsto A^0(U,M)$.
\end{center}
In a future paper, we will show a similar result corresponding to the Nisnevich topology.
\begin{The} \label{Thm6.1}

Let $X$ be an essentially smooth and semi-local scheme, and let be $\VV_X$ a virtual bundle over $X$. Then
\begin{center}

$A^p(X,M,\VV_X)=0$
\end{center}
for $p>0$.
\end{The}
The proof is postponed till the end of the section; we will need the following lemmas.
\par 
Let $V$ be a vector space over $k$ and let $\AAA(V)$ be the associated affine space. For a linear subspace $W$ of $V$ let
\begin{center}

$\pi_W:\AAA(V) \to \AAA(V/W),$
\\ $\pi_W(v)=v+W$
\end{center}
be the projection.

\begin{Lem}\label{Lem6.2}

Let $X\subset \AAA(V)$ be an equidimensional closed subscheme of dimension $d$ and let $Y$ be a closed subscheme  with $\dim Y < d$. Moreover, let $S\subset Y$ be a finite subset such that $X$ is smooth in $S$. Then for a generic $(d-1)$-codimensional linear subspace $W$ of $V$ the following conditions hold.
\begin{itemize}
\item The restriction
$\pi_W|Y:Y\to \AAA(V/W)$
is finite.
\item The restriction $\pi_W|X:X\to \AAA(V/W)$ is locally around $S$ smooth.
\end{itemize}

\end{Lem}
\begin{proof}
See \cite[Proposition 6.2]{Rost96}.
\end{proof}

The existence of the generic space $W$ is not guaranteed over finite base fields. Note that if $X$ is semi-local, then $X\times_k \Spec(k(t))$ is also semi-local and essentially smooth over $k$. The following lemma shows that we can assume $k$ to be infinite.
\begin{Lem}
\label{Lem6.3}
Let $X$ be a scheme over $k$ and let $g:X_{k(t)}\to X$ be the (smooth) base change. Then
\begin{center}

$g^*:A^*(X,M,\VV_X)\to A^*(X_{k(t)},M,-\TT_{X_{k(t)}/X}+\VV_{X_{k(t)}})$
\end{center}
is injective.
\end{Lem}
\begin{proof} (see also \ref{GysinLem5.10})
Let $U\subset \AAA^1_k$ be an open set containing $0$ and write $X^U=X\times_k (U\setminus\{0\})$. Define $s_U^*$ as the composition
\begin{center}

$\xymatrix{
C_{p+1}(X^U,M,-\TT_{X^U/X}+\VV_{X^U}) \ar[d]_{s_U^*} \ar[r]^-{[t]} & C_{p+1}(X^U,M,\VV_{X^U}) \ar[d]^{\partial^U} \\
 C_p(X,M,\VV_{X}) & 
C_p(X\times_k \{0\},M,\VV_{X\times_k \{0\}}). \ar[l]_\simeq
}$

\end{center}
This defines a morphism on the homology groups, also denoted by $s^*_U$. Consider the map ${g^*_U:A_p(X,M,\VV_{X})\to A_{p+1}(X^U,M,-\TT_{X^U/X}+\VV_{X^U})}$ (induced by the canonical projection ${g_U:X^U\to X}$) so that we have $g^*=\colim_U g^*_U$ where the limit is taken over the open sets $U\subset \AAA^1_k$ containing $0$. Define likewise $s^*=\colim_U s_U^*$. Since $s_U^*\circ g_U^*=\id$ for all $U$ (by \ref{itm:R3d}, see also the proof of Lemma \ref{Lem4.5}), we see that $s^*$ is a section of $g^*$.
\end{proof}

\begin{Pro}

\label{Prop6.4}
Let $X$ be a smooth scheme over a field and let $Y\subset X$ be a closed subscheme of codimension $\geq 1$. Then for any finite subset $S\subset Y$ there is an open neighborhood $X'$ of $S$ in $X$ such that the map
\begin{center}

$i_*:A_*(Y\cap X',M, \VV_{Y\cap X'}) \to A_*(X',M,\VV_{X'})$
\end{center}
is the trivial map. Here $i:Y\cap X'\to X'$ is the inclusion.
\end{Pro}
\begin{proof} Same as in \cite[Proposition 6.4]{Rost96}. This uses Lemma \ref{Lem4.5}.
%   Verbatim Rost a priori. on a bien que $T_{Q/Y}\simeq \AAA^1_Q$ canoniquement pour la définition de $H$.
\end{proof}

\begin{proof}[Proof of Theorem \ref{Thm6.1}] We may assume that $X$ is connected. Put $d=\dim X$. Consider pairs $(U,S)$ where $U$ is a smooth $d$-dimensional scheme of finite type over $k$ and $S\subset U$ is a finite subset such that $X$ is the localization of $U$ in $S$. Then
\begin{center}

$C^*(X,M,\VV_X)=\colim_{(U,S)} C^p(U,M,\VV_U)$.
\end{center}
Moreover,
\begin{center}
$C^p(U,M,\VV_U)=C_{d-p}(U,M,\VV_U)=\colim_Y C_{d-p}(Y,M,\VV_Y)$
\end{center}
where $Y$ runs over the closed $p$-codimensional subsets of $U$. Hence
\begin{center}

$A^p(X,M,\VV_X)=\colim_{(U,S)} A^p(U,M,\VV_U)=\colim_{(U,S)} \colim_Y A_{d-p}(Y,M,\VV_Y)$.
\end{center}
Finally, Proposition \ref{Prop6.4} tells that the map $A_{d-p}(Y,M,\VV_Y)\to A^p(U,M,\VV_U)\to A^p(U',M,\VV_{U'})$ is trivial for $U'\subset U$ small enough.
\end{proof}

\begin{Cor}

\label{Cor6.5}
Let $X$ be a smooth scheme over $k$, let $\VV_X$ be a virtual bundle over $X$ and let $\mathcal{M}_X$ be the Zariski sheaf on $X$ given by
\begin{center}

$U\mapsto A^0(U,M,\VV_U) \subset M(\xi_X,\VV_X)$
\end{center}
for open subsets $U$ of $X$. There are natural isomorphisms
\begin{center}

$A^p(X,M,\VV_X)\to H^p_{Zar}(X,\mathcal{M}_X)$.
\end{center}
\end{Cor}
\begin{proof}
Same as \cite[Corollary 6.5]{Rost96}.
\end{proof}

\section{Homotopy Invariance} \label{HomotopyInvarianceSection}
Following \cite[§2]{Deg14}, we define a coniveau spectral sequence that will help us reduce the homotopy invariance property to the known case \ref{itm:(H)}.

\begin{Pro} \label{PropertyH}
Let $F$ be a field, $X=\Spec F$ its spectrum, $\AAA^n_X$ the affine vector bundle of rank $n$ over $X$, $\pi:\AAA^n_X\to X$ the canonical projection and $\VV_X$ a virtual vector bundle over $X$. Then, for every $q\in \ZZ$, the canonical morphism
\begin{center}
$\pi^*:A^q(X,M,\VV_X)\to A^q(\AAA^n_X,M,-\TT_{\AAA^n_X/X}+\VV_{\AAA^n_X})$

\end{center}
is an isomorphism.
\end{Pro}
\begin{proof} Same as \cite[Satz 6.1.1]{Schmid98} or \cite[Théorème 11.2.4]{Fasel13}. If $q \neq 0$, then the map is trivial. We assume $q=0$. By induction, we can reduce to the case $n=1$. In this case, we have $\AAA^1_X=\Spec F[t]$ for some parameter $t$. By definition, $A^0(X,M,\VV_X)=M(F,\Om_{F/k}+\VV_F)$ and $A^0(\AAA^n_X,M,-\TT_{\AAA^1_X/X}+\VV_{\AAA^n_X})$ is the cokernel of the map ${d:M(F(t),\Om_{F(t)/k} -\AAA^1_{F(t)}+\VV_{F(t)})\to \bigoplus_{x\in {(\AAA_F^1)}^{(1)}} M(\kappa(x),\Om_{\kappa(x)/k} -\AAA^1_{\kappa(x)}+\VV_{\kappa(x)})}$ defined in \ref{DefDifferential}. Thus, the result follows from Proposition \ref{Prop2.2}, \ref{itm:(H)}.
\end{proof}

\par Let $X$ be a scheme with a virtual bundle $\VV_X$, let $V$ be a scheme and let $\pi:V\to X$ be an essentially smooth morphism.

A flag on $X$ is a decreasing sequence $(Z^p)_{p\in \ZZ}$ of closed subschemes of $X$ such that $\codim_X(Z^p) \geq p$ where, by convention, $Z^p=X$ if $p< 0$ and $Z^p=\varnothing$ if $p> \dim X$. The set of flags of $X$ is denoted by $\Flag(X)$ and it is ordered by the inclusion termwhise so that it becomes a filtrant set.
\par Let $\mathfrak{Z}=(Z^p)_{p\in \ZZ}$ be a flag of $X$ and define $\pi^*\mathfrak{Z}=(\pi^*Z^p)_{p\in \ZZ}$ a flag over $V$ by ${\pi^*Z^p=V\times_X Z^p}$. For $p,q\in \ZZ$, define
\begin{center}

$D_{\mathfrak{Z}}^{p,q}=A^{p+q-1}(V- \pi^*Z^p, M, \VV_{(X -  \pi^*Z^p)})$,
\\ $E_{\mathfrak{Z}}^{p,q}=A^{q}(\pi^*Z^p- \pi^*Z^{p+1},M,\VV_{(\pi^*Z^p-  \pi^*Z^{p+1})})$.
\end{center}
We have a long exact sequence
\begin{center}

$\xymatrix{
 \dots \ar[r] & D_{\mathfrak{Z}}^{p+1,q-1} \ar[r]^-{j_p^*} & D_{\mathfrak{Z}}^{p,q} \ar[r]^{\partial_p} & E_{\mathfrak{Z}}^{p,q} \ar[r]^{i_{p,*}} & D_{\mathfrak{Z}}^{p+1,q} \ar[r] & \dots
 }$
\end{center}

so that $(D^{p,q}_\mathfrak{Z},E^{p,q}_\mathfrak{Z})_{p,q\in \ZZ}$ is an exact couple where $j_p^*$ and $i_{p,*}$ are induced by the canonical immersions. By the general theory (see \cite[Chapter 3]{McCleary00}), this defines a spectral sequence that converges to $A^{p+q}(V,M,\VV_V)$ because the $E^{p,q}_1$-term is bounded (since the dimension of $V$ is finite).
\par For $p,q\in \ZZ$, denote by
\begin{center}
$D^{p,q}_{1,\pi}=\colim_{\mathfrak{Z}\in \Flag(X)} D^{p,q}_{\mathfrak{Z}}$,\\
$E^{p,q}_{1,\pi}=\displaystyle \operatorname{colim}_{\mathfrak{Z}\in \Flag(X)} E_{\mathfrak{Z}}^{p,q}$
\end{center}
where the colimit is taken over the flags $\mathfrak{Z}$ of $X$ (see Proposition \ref{Prop4.4} for functoriality). Since filtered direct limits are exact in the derived category of abelian groups, the previous spectral sequences give the following theorem.
\begin{The} \label{SpectralSequenceHmtpInvariance}

We have the convergent spectral sequence
\begin{center}

$E^{p,q}_{1,\pi} \Rightarrow A^{p+q}(V,M,\VV_V)$.
\end{center}

\end{The}

We need to compute this spectral sequence. This is done in the following theorem.
\begin{The} \label{SpectralSequenceComputation}

For $p,q\in \ZZ$, we have a canonical isomorphism
\begin{center}

$E^{p,q}_{1,\pi}\simeq \displaystyle \bigoplus_{x\in X^{(p)}}A^q(V_x,M,\VV_{V_x})$.
\end{center}
\end{The}
\begin{proof} (see \cite[§4]{Deg14} for a similar result in the oriented case) Denote by $\mathcal{I}_p$ the set of pairs $(Z,Z')$ where $Z$ is a reduced closed subscheme of $X$ of codimension $p$ and $Z'\subset Z$ is a closed subset containing the singular locus of $Z$. Notice that any such pair $(Z,Z')$ can be (functorially) extented into a flag of $X$. Moreover, for any $x$ in $X$, consider $\overline{ \{x\}}$ the reduced closure of $x$ in $X$ and $\mathfrak{F}(x)$ be the set of closed subschemes $Z'$ of $\overline{ \{x\}}$ containing its singular locus. The following equalities are all canonical isomorphisms:

\begin{tabular}{lll}

$E^{p,q}_{1,\pi}$ &$ \simeq$ & $\colim_{{\mathfrak{Z}}\in \Flag(X)} A^q(V\times_X(Z^p-Z^{p+1}),M,\VV_{(V\times_X(Z^p-Z^{p+1}))})$ \\
    &    $\simeq$  & $\colim_{(Z,Z')\in \mathcal{I}_p} A^q(V\times_X(Z-Z'),M,\VV_{(V\times_X(Z-Z'))})$ \\
    &   $\simeq$   & $\displaystyle \bigoplus_{x\in X^{(p)}} \colim_{Z'\in \mathfrak{F}(x)} A^q(V\times_X(\overline{ \{x\} }-Z'),M,\VV_{V\times_X(\overline{ \{x\} }-Z')})$ \\
    &  $\simeq$    & $\displaystyle \bigoplus_{x\in X^{(p)}} A^q(V_x,M,\VV_{V_x}).$
\end{tabular}

\end{proof}

\begin{The}[Homotopy Invariance] \label{HomotopyInvariance}
Let $X$ be a scheme, $V$ a vector bundle of rank $n$ over $X$, $\pi:V\to X$ the canonical projection and $\VV_X$ a virtual vector bundle over $X$. Then, for every $q\in \ZZ$, the canonical morphism
\begin{center}
$\pi^*:A^q(X,M,\VV_X)\to A^q(V,M,-\TT_{V/X}+\VV_V)$

\end{center}
is an isomorphism.
\end{The}

\begin{proof}

% First, remark that $\TT_{V/X}$ is canonically isomorphic to $\AAA^n_V$ so that it behaves well with pullbacks. 
From a noetherian induction and the localization sequence \ref{LocalizationSequence}, we can reduce to the case where $V=\AAA^n_X$ is the affine vector bundle of rank $n$. With the previous notations, Theorem \ref{SpectralSequenceHmtpInvariance} gives the spectral sequence
\begin{center}

$E^{p,q}_{1,\pi} \Rightarrow A^{p+q}(V,M,- \TT_{V/X}+\VV_V)$
\end{center}
where $E^{p,q}_{1,\pi} $ is (abusively) defined as previously, but twisted accordingly.
By Theorem \ref{SpectralSequenceComputation}, the page $E^{p,q}_{1,\pi}$ is isomorphic to $\bigoplus_{x\in X^{(p)}}A^q(V_x,M,-\TT_{V_x/X_x}+\VV_{V_x})$. According to Proposition \ref{PropertyH}, this last expression is isomorphic (via the map $\pi$) to $\bigoplus_{x\in X^{(p)}}A^q(\Spec \kappa(x),M,\VV_x)$ (in other words, the theorem is true when $X$ is the spectrum of a field). Using again \ref{SpectralSequenceComputation}, this group is isomorphic to $E^{p,q}_{1,\Id_X}$, which converge to $A^{p+q}(X,M,\VV_X)$. By Proposition \ref{Prop4.4}, the map $\pi$ induces a morphism of exact couples $(D^{p,q}_{1,\Id_X},E^{p,q}_{1,\Id_X})\to (D^{p,q}_{1,\pi},E^{p,q}_{1,\pi})$ hence we have compatible isomorphisms on the pages which induce the pullback ${\pi^*:A^q(X,M,\VV_X)\to A^q(V,M,-\TT_{V/X}+\VV_V)}$ (see also \cite[5.2.12]{Weibel}).

\end{proof}

\section{Gysin Morphisms}\label{GysinMorphisms}

\subsection*{Gysin morphisms for regular embeddings}\label{GyMoRegEmbedding}
We define Gysin morphisms for regular closed immersions and prove functoriality theorems. As always, the main tool is the deformation to the normal cone.

\par Let $i:Z\to X$ be a closed immersion, and let $\VV_X$ be a virtual bundle over the scheme $X$. Let $t$ be a parameter such that $\AAA_k=\Spec k[t]$, and let ${q:X\times_k (\AAA^1_k\setminus \{0\}) \to X}$ be the canonical projection. Denote by $D=D_ZX$ the deformation space such that ${D=U \sqcup N_ZX}$ where $U=X\times_k(\AAA^1_k\setminus \{0\})$ (see \cite[§10]{Rost96} for more).

Consider the morphism 
\begin{center}

$J(X,Z)=J_{Z/X}:C_*(X,M,\VV_X)\to C_*(N_ZX,M,\VV_{N_ZX})$

\end{center} defined by the composition:
\begin{center}

$
\xymatrix{
C_{p}(X,M,\VV_X) \ar[r]^-{q^*} \ar[d]^{J_{Z/X}} & C_{p+1}(U,M,-\TT_{U/X}+\VV_U) \ar[d]^{[t]} \\ C_{p}(N_ZX,M,\VV_{N_ZX})  & C_{p+1}(U,M,\VV_U) \ar[l]^-\partial 
}
$

\end{center}
where the multiplication with $t$ is twisted by the isomorphism $\TT_{U/X}\simeq \AAA^1_U$ (which only depends on $t$) and $\partial$ is the boundary map as in \ref{BoundaryMaps}. This defines a morphism (also denoted by $J(X,Z)$ or $J_{Z/X}$) by passing to homology.
\par 
Assume moreover that $i:Z\to X$ is regular of codimension $m$, the map $\pi : N_ZX \to Z$ is a vector bundle over $X$ of dimension $m$. By homotopy invariance (Theorem \ref{HomotopyInvariance}), we have an isomorphism
\begin{center}

$\pi^*:A_p(Z,M,\NN_ZX+\VV_Z) \to A_{p+m}(N_ZX,M,\VV_{N_ZX})$
\end{center}
where $\NN_ZX$ is the virtual vector bundle associated to $N_ZX$
(we have used the canonical isomorphism $\pi^*(\NN_ZX)=\TT_{N_ZX/Z}$). Denote by $r_{Z/X}=(\pi^*)^{-1}$ its inverse.

\begin{Def}

With the previous notations, we define the map
\begin{center}
$i^*:A_p(X,M,\VV_X)\to A_{p-m}(Z,M,\NN_ZX+\VV_Z)$

\end{center}
by putting $i^*=r_{Z/X}\circ J_{Z/X} $ and call it the {\em Gysin morphism of $i$}.
\end{Def}

The following lemmas are needed to prove functoriality of the previous construction (see Theorem \ref{GysinFunctoriality}).

\begin{Lem}\label{Lem11.3}

Let $i:Z\to X$ be a regular closed immersion and $g:V\to X$ be an essentially smooth morphism. Denote by $N(g)$ the projection from $N(V,{V\times_X Z}) =N_Z(X)\times_X V$ to $N_ZX$. Then
\begin{center}

$\Theta\circ J(V,V\times_X Z)\circ g^*=N(g)^*\circ J(X,Z)$,
\end{center}
where $\Theta$ comes from the canonical isomorphism $\TT_{N(V,Y\times_X V)/N(X,Y)}\simeq T_{V/X}\times_V N(V,Y\times_X V)$.
\end{Lem}
\begin{proof}
See \cite[Lemma 11.3]{Rost96}. This follows from our Proposition \ref{Prop4.1}.2, Proposition \ref{Prop4.4}.2 and Proposition \ref{Lem4.3}.2.
\end{proof}

\begin{Lem}\label{Lem11.4}
Let $Z\to X$ be a closed immersion and let $p:X\to Y$ be essentially smooth. Let $\VV_Y$ be a virtual vector bundle over $Y$. Suppose that the composite 
\begin{center}

$q:N_ZX\to Z \to X \to Y$
\end{center}
is essentially smooth of same relative dimension as $p$. Then 
\begin{center}

$\Theta\circ J(X,Z)\circ p^*=q^*$
\end{center}
where $\Theta$ comes from the canonical isomorphism $\TT_{N_ZX/Y}\simeq  \TT_{V/Y}\times_X N_ZX$.
\end{Lem}
\begin{proof} Same as \cite[Lemma 11.4]{Rost96} except that Rost only needs the composite morphism
\begin{center}
$\xymatrix{
f:D(X,Z) \ar[r] &  X\times_k \AAA^1_k \ar[r]^{p\times \Id} & Y\times_k \AAA^1_k
}$ 
\end{center} to be flat. We need moreover (in order to use our Proposition \ref{Lem4.5}) the fact that $f$ is essentially smooth which is true because it is flat and its fibers are essentially smooth.
\end{proof}
\begin{Lem}\label{Lem11.6}
Let $\rho:T\to T'$ be a morphism, let $T_1',T_2'\subset T'$ be closed subschemes and let $T_i=T\times_{T'} T_i'$ for $i=1,2$. Let $\VV_{T'}$ be a virtual bundle vector bundle over $T'$.
\par Put $T_3=T\setminus (T_1\cup T_2)$, $T_0=T_1\subset T_2$, $\tilde{T_1}=T_1\setminus T_0$, $\tilde{T_2}=T_2\setminus T_0$ and let $\partial^3_1, \partial^1_0, \partial^3_2, \partial^2_0$ be the boundary maps for the closed immersions
\begin{center}

$\tilde{T_1}\to T \setminus T_2$, $T_0\to T_1$, $\tilde{T_2}\to T\setminus T_1$, $T_0\to T_2$,
\end{center}
respectively. Then
\begin{center}

$0= \partial^1_0\circ \partial^3_1+\partial^2_0\circ \partial^3_2:[T_3,\VV_{T_3}]\bullet\!\!\! \to [T_0,\VV_{T_0}]$
\end{center}
at the homology level (recall the notation introduced in \ref{GeneralizedCorr}).
\end{Lem}
\begin{proof}

Corresponding to the set theoretic decomposition of $T$ we have
\begin{center}

$A_*(T,M,\VV_T)=A_*(T_0,M,\VV_{T_0})\oplus A_*(\tilde{T_1},M,\VV_{\tilde{T_1}}) \oplus  A_*(\tilde{T_2},M,\VV_{\tilde{T_2}})\oplus A_*(T_3,M,\VV_{T_3}).$
\end{center}
Then $d_T\circ d_T=0$ gives the result.
\end{proof}

\begin{Lem}\label{Lem11.7}
Let $T=\overline{ D}=\overline{ D}(X,Y,Z)$, $T_1=\overline{ D}|(\{0\}\times \AAA^1_k)$, $T_2=\overline{ D}(\AAA^1_k\times \{0\})$ where $\overline{ D}$ is the double deformation cone (see \cite[§10.5]{Rost96}). We keep the notations of Lemma \ref{Lem11.6}. Then ${T_3=X\times_k (\AAA^1_k\setminus \{0\})\times (\AAA^1_k\setminus \{0\})}$ and $T_0=\overline{ D}|\{0,0\}$. Let $\pi:T_3\to X$ be the projection and let $t,s$ be the coordinates of $\AAA^2_k=\Spec k[t,s]$, so that $T_1=\{t=0\}$, $T_2=\{s=0\}$.
\par Let $Z\to Y\to X$ be regular closed immersions. Then (up to canonical isomorphisms $\Theta$ and $\Theta'$)
\begin{center}

$\partial^1_0\circ \partial^3_1\circ [t,s]\circ \pi^* =
\Theta \circ J(N_YX,N_YX|Z)\circ J(X,Y)$,\\
$\partial^2_0\circ \partial^3_2\circ [s,t] \circ \pi^* = \Theta' \circ J(N_ZX,N_ZY)\circ J(X,Z)$.
\end{center}
\end{Lem}
\begin{proof}

Same as \cite[Lemma 11.7]{Rost96} except one has to be careful with the twists (this uses Lemma \ref{Lem11.3}).
\end{proof}

\begin{The} \label{Thm13.1} \label{GysinFunctoriality}
Let $l:Z\to Y$ and $i:Y\to X$ be regular closed immersions of respective codimension $n$ and $m$. Then $i\circ l$ is a regular closed immersion of codimension $m+n$ and (up to a canonical isomorphism)
\begin{center}

$(i\circ l)^*= l^*\circ i^*$ as morphism $A_p(X,M,\VV_X)\to A_{p-m-n}(Z,M,\NN_ZX+\VV_Z)$.
\end{center}
\end{The}
\begin{proof}
The first assertion follows from \cite[19.1.5.(iii)]{EGA4} which also gives the canonical isomorphism of virtual vector spaces $\NN_ZX=\NN_ZY+\NN_YX\times_Y Z$. The second assertion follows from \ref{Lem11.3}, \ref{Lem11.6}, \ref{Lem11.7} and \ref{Lem11.4} as in \cite[Theorem 13.1]{Rost96}.
\end{proof}

We conclude with one interesting property.

%  Proposition à inclure plus tard 
%  \begin{Pro} 
%  
%  
%  
%  Let $h:X\to X'$ be a morphism of schemes. Let $Z'\hookrightarrow X'$ be a closed immersion. Consider the induced diagram given by $U'=X'\setminus Y'$ and pullback:
%  
%  
%  
%  \begin{center}
%  $\xymatrix{
%  Z \ar@{^{(}->}[r] \ar[d]^{f} & X \ar[d]^h & U \ar@{_{(}->}[l] \ar[d]^{g} \\
%  Z' \ar@{^{(}->}[r] & X'  & U' \ar@{_{(}->}[l]
%  }$
%  \end{center}
%  
%  If $h$ is a regular closed immersion, then
%  \begin{center}
%  
%  $f^*\circ \partial^{U'}_{Y'} = \partial^U_Y \circ g^*.$
%  \end{center}
%  
%  
%  \end{Pro}
%  \begin{proof}
%  This follows from the fact that $h^*$ is a complex morphism.
%  \end{proof}
%  

\begin{Pro}[Base change for regular closed immersions]
 \label{BaseChangeRegular}
Consider the cartesian square
\begin{center}
$\xymatrix{
Z \ar[r]^i \ar[d]^g & X \ar[d]^f \\
Z'\ar[r]^{i'} & X'
}$
\end{center}
where $i,i'$ are regular closed immersions and $f,g$ are proper morphisms. Suppose moreover that we have a canonical isomorphism of virtual vector spaces $\NN_{Z'}X' \simeq \NN_ZX\times_X X'$. Then (up to the canonical isomorphism $\Theta$ induced by the previous isomorphism)
\begin{center}

$\Theta \circ g_*\circ i^* = i'^*\circ f_*$.
\end{center}
\end{Pro}
\begin{proof}
It suffices to prove that the following diagram is commutative (recall the notation introduced in \ref{GeneralizedCorr}):
\begin{center}

$\xymatrixcolsep{4pc}\xymatrix{
X \ar@{{*}->}[r]^-{q^*} \ar@{{*}->}[d]^{f_*} \ar@{}[rd]|-{(1)} &  X\times (\AAA^1_k\setminus\{0\}) \ar@{{*}->}[r]^{[t]} \ar@{{*}->}[d]^{(f\times \Id)_*} \ar@{}[rd]|-{(2)} & 
 X\times (\AAA^1_k\setminus\{0\})  \ar@{{*}->}[r]^-{\partial} \ar@{{*}->}[d]^{(f\times \Id)_*} \ar@{}[rd]|-{(3)} & \NN_ZX   \ar@{{*}->}[r]^-{\simeq} \ar@{{*}->}[d]^{h_*} \ar@{}[rd]|-{(4)} & Z \ar@{{*}->}[d]^{g_*} \\
 X' \ar@{{*}->}[r]^-{q'^*} &  X'\times (\AAA^1_k\setminus\{0\}) \ar@{{*}->}[r]^{[t]}  & 
  X'\times (\AAA^1_k\setminus\{0\})  \ar@{{*}->}[r]^-{\partial'} & \NN_{Z'}X'   \ar@{{*}->}[r]^-{\simeq} & Z'
}$
\end{center}
with obvious notations (see the definition of Gysin morphisms).
The (cartesian) squares (1) and (4) commute by the base change theorem for essentially smooth morphisms (see Proposition \ref{Prop4.1}.3). The squares (2) and (3) commute by Proposition \ref{Lem4.2}.1 and Proposition \ref{Prop4.4}.1, respectively.
\end{proof}
\subsection*{Gysin morphisms for lci projective morphisms}
We define Gysin morphisms for lci projective morphisms\footnote{By definition, a map $f:Y\to X$ between two schemes is called a {\em local complete intersection projective} morphism (or {\em lci projective}) if there exist a natural number $n$ and a factorization 
$\xymatrix{ Y \ar[r]^i & \PP^n_X \ar[r]^p & X}$ of $f$ into a regular closed immersion followed by the canonical projection.} and prove functoriality theorems (see \cite[§5]{Deg08n2} for similar results in the classical oriented case). One could also define Gysin morphisms for morphisms between two essentially smooth schemes as in \cite[§12]{Rost96}.
                 %\begin{Lem}
                 %
                 %Let $n,m$ be natural numbers and $X$ be an essentially smooth scheme. Consider the essentially smooth projection morphisms
                 %\begin{center}
                 %
                 %$\xymatrix{
                 %\PP_X^n\times_X \PP_X^m \ar[r]_{q'} \ar[d]^{p'} & \PP_X^n \ar[d]^p \\
                 %\PP_X^m \ar[r]^q & X.
                 %}$
                 %\end{center}
                 %Then $q'^*p^*=p'^*q^*$.
                 %\end{Lem}
                 %\begin{proof}
                 %This is a particular case of \ref{Prop4.1}.2).
                 %\end{proof}

\begin{Lem}\label{GysinLem5.9}

Consider a regular closed immersion $i:Z\to X$ and a natural number $n$. Consider the pullback square

\begin{center}

$\xymatrix{
\PP^n_Z \ar[r]^l \ar[d]^q & \PP^n_X \ar[d]^p \\
Z \ar[r]^i & X.
}$
\end{center}
Then $l^*\circ p^*=q^*\circ i^*$ (up to the canonical isomorphism induced by $\TT_q-q^*\NN_i\simeq l^*\TT_p-\NN_l$).
\end{Lem}
\begin{proof}

This follows from the definitions and Lemma \ref{Lem11.3}.
\end{proof}

\begin{Lem}\label{GysinLem5.10}
Consider a natural number $n$ and an essentially smooth scheme $X$. Let $p:\PP_X^n\to X$ be the canonical projection. Then for any section $s:X\to \PP^n_X$ of $p$, we have $s^*p^*=\Id$ (up to a canonical isomorphism).
\end{Lem}
\begin{proof}
We can assume that $X=\Spec k$, then apply rule \ref{itm:R3d} (see also the proof of Lemma \ref{Lem4.5}).
\end{proof}

\begin{Lem} \label{GysinLem5.11}
Consider the following commutative diagram:
\begin{center}

$\xymatrix{
 & \PP^n_X \ar[rd]^p & \\
 Y \ar[ru]^i \ar[rd]_{i'} & & X \\
  & \PP^m_X \ar[ru]_q & 
  }$
\end{center}
where $i,i'$ are regular closed immersions and $p,q$ are the canonical projection. Then $i^*\circ p^*={i'}^*\circ q^*$ (up to the canonical isomorphism induced by $\NN_{i'}-(i')^*\TT_q \simeq \NN_i - i^*\TT_p$).
\end{Lem}
\begin{proof} Let us introduce the following morphisms:
\begin{center}

$\xymatrix{
  &   &   \PP_X^n \ar[rd]^p&    \\
  Y \ar@/^/[rru]^i \ar[r]|-\nu \ar@/_/[rrd]_{i'} & \PP^n_X\times_X \PP^m_X \ar[ru]_{q'} \ar[rd]^{p'} & & X \\
   &   &   \PP^m_X \ar[ru]_q &   
   }$
\end{center}
Applying Proposition \ref{Prop4.1}.2, we are reduced to prove $i^*=\nu^*q'^*$ and ${i'}^*=\nu^*p'^*$. In other words, we are reduced to the case $m=0$ and $q=\Id_X$.
\par In this case, we introduce the following morphisms:
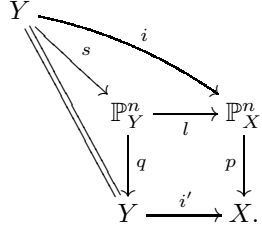
\begin{center}

$\xymatrix{
Y  \ar@/^/^i[rrd]  \ar[rd]^s \ar@2{-}[rdd] &      &        \\
  & \PP^n_Y \ar[r]_l \ar[d]^q & \PP^n_X  \ar[d]_p \\
  & Y     \ar[r]^{i'}  &  X.    
  }$
\end{center}
Then the lemma follows from Lemma \ref{GysinLem5.9}, Lemma \ref{GysinLem5.10} and Theorem \ref{Thm13.1}.
\end{proof}

Let $X$ and $Y$ be two schemes and let $f:Y\to X$ be a projective lci morphism. Consider an arbitrary factorization 
$\xymatrix{ Y \ar[r]^i & \PP^n_X \ar[r]^p & X}$ of $f$ into a regular closed immersion followed by the canonical projection. By the preceding lemma, the composition $i^*p^*$ does not depend on the chosen factorization.

\par 

We recall the definition of $\TT_f$, the virtual tangent bundle of $f$ (see \cite[§3]{Fasel18bis}; see also \cite[B.7.6]{Fult}). Consider a factorization $\xymatrix{ Y \ar[r]^i & \PP^n_X \ar[r]^p & X}$ as before and define $\TT_{pi}=i^*\TT_{\PP^n_X/X} - \NN_Y(\PP^n_X) $ the relative (virtual) bundle with respect to the factorization $f=pi$. Consider next a commutative diagram
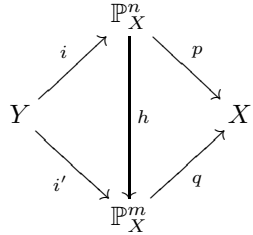
\begin{center}

$\xymatrix{
 & \PP^n_X \ar[dd]^h\ar[rd]^p & \\
 Y \ar[ru]^i \ar[rd]_{i'} & & X \\
  & \PP^m_X \ar[ru]_q & 
  }$
\end{center}
where $pi=qi'=f$ are two factorizations of $f$ and where $h$ is smooth. By taking the pullback, we can construct a third factorization $f=\pi(i\times i')$ such that we have canonical isomorphisms of virtual vector bundles
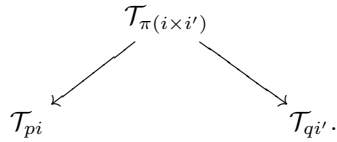
\begin{center}

$\xymatrix{
 & \TT_{\pi(i\times i')} \ar[ld] \ar[rd] & \\
 \TT_{pi} & & \TT_{qi'}.
 }$
\end{center}
Hence we can define $\TT_f$ as the limit (over all factorizations) of $\TT_{pi}$.
%     Note that, when $X$ and $Y$ are essentially smooth, we have $\TT_f=\TT_{X/Y}$ as expected.
\begin{Def}
Keeping the previous notations, we define the Gysin morphism associated to $f$ as the morphism 
\begin{center}

$f^*=\Theta \circ i^*\circ p^*:A^*(X,M,\VV_X)\to A^*(Y,M,-\TT_{f}+\VV_Y)$
\end{center}
where $\Theta$ is the canonical isomorphism coming from $\TT_f\simeq \TT_{pi}$.
\end{Def}

\begin{Pro}

Consider projective morphisms $\xymatrix{Z \ar[r]^g & Y \ar[r]^f & X}$.
Then (up to the canonical isomorphim induced by $\TT_{fg}\simeq \TT_g + g^*\TT_f$):
\begin{center}
$g^*\circ f^*=(f\circ g)^*$.

\end{center} 
\end{Pro}
\begin{proof} (see also \cite[Proposition 5.14]{Deg08n2} for the classical oriented case)

We choose a factorization $\xymatrix{Y \ar[r]^i & \PP^n_X \ar[r]^p & X}$ 
(resp. $\xymatrix{Z\ar[r]^j & \PP^m_X \ar[r]^q & X}$) of $f$ (resp. $fg$) and we introduce the diagram
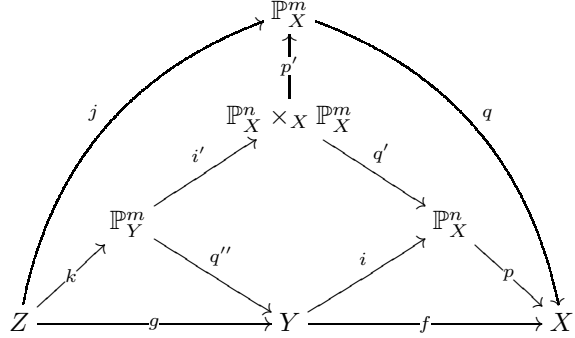
\begin{center}

$\xymatrix{
    &     &    \PP^m_X  \ar@/^2pc/	[rrddd]^q &    &   \\
    &      &   \PP^n_X\times_X \PP^m_X \ar[u]|-{p'} \ar[rd]^{q'} &   &   \\
    &  \PP^m_Y \ar[rd]^{q''} \ar[ru]^{i'}&                 &   \PP^n_X \ar[rd]|-{p}&   \\
    Z \ar[rr]|-{g} \ar[ru]|-{k}  \ar@/^2pc/[rruuu]^j &  &      Y   \ar[ru]^i      \ar[rr]|-{f}    &        &  X
    }$ 
\end{center}
in which $p'$ is deduced from $p$ by base change, and so on for $q'$ and $q''$. Then, by using the factorization given in the preceding diagram, the proposition follows from \ref{GysinLem5.9}, \ref{Thm13.1}, \ref{GysinLem5.11} and \ref{Prop4.1}.2.
\end{proof}

Now consider a cartesian square of schemes
\begin{center}

$\xymatrix{
X' \ar[r]^{f'} \ar[d]_{g'} & Y' \ar[d]^g \\
X \ar[r]_f & Y
}$
\end{center}
with $f$ proper 
%and essentially smooth
 and $g$ lci.
Suppose morever that the square is {\em tor-independent}, that is for any $x\in X$, $y'\in Y'$ with $y=f(x)=g(y')$ and for any $i>0$ we have
\begin{center}

$\operatorname{Tor}^{\mathcal{O}_{Y,y}}_i(\mathcal{O}_{X,x},\mathcal{O}_{Y',y'})=0$.
\end{center}
From the fact that cotangent complex is stable under derived base change,
%(see Stack Project tag 08QQ)
it follows that there is a canonical isomorphism $f^{'*}\TT_g \simeq \TT_{g'}$ so that the following proposition makes sense.

\begin{Pro}[Base change for lci morphisms]\label{BaseChangelci} Keeping the previous assumptions, we have (up to the canonical isomorphism induced by the previous isomorphism):
\begin{center}
$f'_* \circ g'^* = g^*\circ f_*$.
\end{center}
\end{Pro}
\begin{proof}

It suffices to consider the case where $g$ is the projection of a projective bundle or a regular closed immersion. It follows from Proposition \ref{Prop4.1} in the first case and from Proposition \ref{BaseChangeRegular} in the second.
\end{proof}

\section{Products} \label{ProductSection}

\begin{Par}{\sc Cross products} Let $M\times M' \to M''$ be a pairing of MW-cycle modules over $k$. Let $Y$ and $Z$ be two essentially smooth schemes over $k$ equipped with virtual vector bundles $\VV_Y$ and $\mathcal{W}_Z$. We define the cross product
\begin{center}

$\times : C_p(Y,M,\VV_Y)\times C_q(Z,M',\mathcal{W}_Z)\to C_{p+q}(Y\times Z,M'',\VV_{Y\times Z}+\mathcal{W}_{Y\times Z})$
\end{center}
as follows. For $y\in Y$, let $Z_y=\Spec \kappa(y) \times Z$, let $\pi_y:Z_y\to Z$ be the projection and let $i_y:Z_y\to Y\times Z$ be the inclusion. For $z\in Z$ we understand similarly $Y_z,\pi_z:Y_z\to Y$ and $i_z:Y_z\to Y\times Z$. We give the following two equivalent definitions:
\begin{center}
$\rho\times \mu = \displaystyle \sum_{y\in Y_{(p)}} (i_y)_*(\rho_y\cdot \pi_y^*(\mu))$,
\\ $\rho\times \mu = \displaystyle \sum_{z\in Z_{(q)}} (i_z)_*(\pi^*_z(\rho)\cdot \mu)$.
\end{center}
We give more details on this definition. The map 
\begin{center}
${(i_y)_*: C_q(Z_y,M,\VV_{Z_y}+\mathcal{W}_{Z_y})\to C_{p+q}(Y\times Z,M, \VV_{Y\times Z} +\mathcal{W}_{Y\times Z})}$ 

\end{center}is by definition the inclusion corresponding to $Z_{y(q)} \subset (Y\times Z)_{(p+q)}$. Let $y\in Y$ and denote by $\rho_y$ the $y$-component of $\rho$ in $M(\kappa(y), \Omega_{\kappa(y)/k}+\VV_y)$. Write $\pi^*_y(\mu)=\sum_{u\in Z_{y(q)}} (\pi^*_y(\mu))_u$. By definition, 
\begin{center}
$\rho_y\cdot \pi^*_y(\mu)=\displaystyle \sum_{u\in Z_{y(q)}}\Theta_u(\res_{\kappa(u)/\kappa(y)}(\rho_y)\cdot (\pi^*_y(\mu))_u)$,

\end{center}
where $\Theta_u$ is the canonical isomorphism induced by ${\TT_{Z_y/Z}\times \Spec \kappa(u) \simeq \TT_{\kappa(y)/k}\times \Spec \kappa(u)}$. 
\par To check equality of the two definitions, consider the $u$-component for $u\in Y\times Z$. Let $y,z$ be the images of $u$ under the projections $Y\times Z\to Y,Z$. The $u$-components are either trivial or given by
\begin{center}

$(\rho\times \mu)_u=\res_{\kappa(u)/\kappa(y)}(\rho_y)\cdot \res_{\kappa(u)/\kappa(z)}(\mu_z)$.
\end{center}
\end{Par}
\begin{Par}

{\sc Associativity}
 Consider four pairings 
\begin{center}

$M\times M'\to N$,
\\ $M'\times M''\to N'$,
\\ $N\times M''\to N''$,
\\ $M\times N'\to N''$.
\end{center}
For $X$ another scheme and $\nu \in C_r(X,M,\VV_X)$, we have the identity
\begin{center}

$\nu \times ( \rho \times \mu)=(\nu \times \rho) \times \mu$.
\end{center}
Indeed, consider the $u$-component for $u\in X\times Y \times Z$. Let $x,y,z$ be the images of $u$ in $X,Y,Z$, respectively. The $u$-component is either trivial or given by
\begin{center}

$(\nu \times \rho \times \mu)_u=\res_{\kappa(u)/\kappa(x)}(\nu_x)\cdot \res_{\kappa(u)/\kappa(y)}(\rho_y) \cdot \res_{\kappa(u)/\kappa(z)}(\mu_z)$
\end{center}
where the right hand of the equality makes sense if we assume the pairings to be compatible in some obvious sense.

\end{Par}

\begin{Par}{\sc Commutativity}
Let $M$ be a MW-cycle module with a ring structure. Let $\tau:Y\times Z \to Z \times Y$ be the interchange of factors. For $\rho\in C_p(Y,M,\VV_Y)$ and $\mu\in C_q(Z,M,\mathcal{W}_Z)$ one has
\begin{center}

$\tau_*(\rho \times \mu)=(-1)^{mn}\Theta(\mu\times \rho) \in C_{p+q}(Z\times Y,M,\VV_{Z\times Y} + \mathcal{W}_{Z\times Y})$
\end{center}
where $m$ and $n$ are the rank of $\VV_Y$ and $\mathcal{W}_Z$ (respectively) and where $\Theta$ is induced by the switch isomorphism $\mathcal{W}_{Z\times Y}+\VV_{Z\times Y} \simeq \VV_{Z\times Y} + \mathcal{W}_{Z\times Y}$.
\par This is immediate from the definitions.
\end{Par}

\begin{Par}

{\sc Chain rule} For $\rho \in C_p(Y,M,\VV_Y)$ and $\mu\in C_q(Z,M,\mathcal{W}_Z)$ one has
\begin{center}

$d(\rho \times \mu)=d(\rho)\times \mu + \epsilon^{n}\rho \times d(\mu)$
\end{center}
where $n=\rk \VV_Y$.
The proof comes from \ref{itm:R3a}, \ref{itm:R3c}, \ref{itm:R3d} and \ref{itm:P3} (see also \cite[§14.4]{Rost96} for more details).
\end{Par}

\begin{Par}

{\sc Intersection} For $X$ smooth, the product induces a map
\begin{center}

$A^p(X,M,\VV_X)\times A^q(X,M,\mathcal{W}_X)\to A^{p+q}(X\times X,M,\VV_{X\times X}+\mathcal{W}_{X\times X})$.
\end{center}

By composing with the Gysin morphism 
\begin{center}

$\Delta^*:A^{p+q}(X\times X,M,\VV_{X\times X}+\mathcal{W}_{X\times X})\to A^{p+q}(X,M,-\TT_{\Delta}+\VV_X+\mathcal{W}_X)$
\end{center}
induced by the diagonal $\Delta:X\to X\times X$, we obtain the map
\begin{center}
$A^p(X,M,\VV_X)\times A^q(X,M,\mathcal{W}_X)\to A^{p+q}(X,M,-\TT_{\Delta}+\VV_X+\mathcal{W}_X)$.

\end{center}

\end{Par}
The preceding considerations and the functoriality of the Gysin maps prove the following theorem.
\begin{The}

If $M$ is a MW-cycle module with a ring structure and $X$ a smooth scheme over $k$, the pairing turns $\bigoplus_{\VV_X\in \mathfrak{V}(X)}A^*(X,M,\VV_X)$ into a graded commutative associative algebra over $\bigoplus_{\VV_X\in \mathfrak{V}(X)}A^*(X,\KMW,\VV_X)$ (see also \ref{RingStructure} for another formulation).
\end{The}

In particular, we obtain a product on $\CHt(X)$ which coincides with the intersection product (defined in \cite[§3.4]{Fasel18bis}, see also \cite{Fasel13}). Indeed, our construction of Gysin morphisms follows the classical one (using deformation to the normal cone) and our cross products correspond to the one already defined for the Milnor-Witt K-theory (see \cite[§3]{Fasel18bis}).

\section{Adjunction between MW-Cycle Modules and Rost Cycle Modules} \label{Adjunction}

\begin{Par}

Denote by $\mathbf{F}_k$ the category whose objects are couple $(E,n)$ where $E$ is a (finitely generated) field over $k$ and $n$ an integer, and where a morphism $(E,m)\to (F,n)$ is the data of a field extension $E\to F$ and the identity $m=n$.
\par  Consider $M$ a classical cycle module (\textit{à la Rost}, see \cite[§1]{Rost96}). By definition, $M$ is a functor from $\mathbf{F}_k$ to the category of abelian groups with some data $(d1),\dots, (d4)$ and rules $(r1a),\dots, (r3e), (fd)$ and $(c)$ (we use small letters to designate Rost's axioms so that $(d1)$ corresponds to \ref{itm:D1}, etc). Equivalently, $M$ is a functor from the category of finitely generated fields to the category of {\em $\ZZ$-graded} abelian groups satisfying the same conditions.
\par We define a Milnor-Witt cycle module $\Gamma_*(M)$ as follows. Let $(E,\VV_E)$ be in $\mathfrak{F}_k$ and put
\begin{center}

$\Gamma_*(M)(E,\VV_E)=M(E,\rk \VV_E)$.
\end{center}
The data \ref{itm:D1},\dots, \ref{itm:D4} are defined in an obvious way (for \ref{itm:D2} notice that $\rk \Omega_{F/k}=\rk \Om_{E/k}$ if $F/E$ is finite ; for \ref{itm:D3} the generator $\eeta$ acts trivially).
\par Moreover, we can check that this defines a fully faithful exact functor
\begin{center}

$\Gamma_*:\mathfrak{M}^{\operatorname{M}}_k\to \mathfrak{M}^{\operatorname{MW}}_k$.
\end{center}
where $\mathfrak{M}^{\operatorname{M}}_k$ (resp. $\mathfrak{M}^{\operatorname{MW}}_k$) is the category of Rost cycle modules (resp. Milnor-Witt cycle modules). 
\begin{Rem}

 Let $\VV_X$ be a virtual vector bundle of rank $n$ over a scheme $X$ and let $p$ be an integer. With the notations of Section \ref{MWComplexe} and the one in \cite[§5]{Rost96}, we have
\begin{center}

$C_p(X,\Gamma_*(M),\VV_X)=C_p(X,M, n)$.
\end{center}
\end{Rem}
\end{Par}

In order to define the adjoint functor, we need to following lemma.

\begin{Lem}\label{LemAdjunction1} 

 Let $(E,\VV_E)$ be in $\mathfrak{F}_k$ and let $\Theta:\VV_E \to n\cdot \AAA^1_E$, $\Theta':\VV_E \to n\cdot \AAA^1_E$ be two trivializations. The induced isomorphisms
 \begin{center}
 
 $\Theta_*:M(E,\VV_E)\to M(E,n\cdot \AAA^1_E)$, \\
 $\Theta'_*:M(E,\VV_E)\to M(E,n\cdot \AAA^1_E)$ 
 \end{center}
 are related as follows:
 \begin{center}
 
 $(\Theta - \Theta')_*(M(E,\VV_E)) \subset \eeta[u]\cdot M(E, n\cdot \AAA^1_E)$,
 \end{center}
 for some $u\in E^\times$. Hence, they are equal modulo $\eeta$.
 
\end{Lem}
\begin{proof} By \ref{itm:R1a}, write $(\Theta-\Theta')_*=((\Theta\circ \Theta'^{-1})_*-\Id)\circ \Theta'_*$. The rule \ref{itm:R4a} gives an element $u\in E^\times$ such that 
\begin{center}

$(\Theta\circ \Theta'^{-1})_*=\gamma_{\langle u \rangle}$.
\end{center}
Since $\eeta[u]=\langle u \rangle - 1$, the result follows.

\end{proof}
\begin{Par}

Let $M$ be a Milnor-Witt cycle module. We want to define a cycle module \textit{à la Rost} $\Gamma^*(M)$. If $E$ is a field over $k$ and $n$ an integer, denote by $\mathcal{M}_n(E)=M(E,n\cdot \AAA^1_E)$ and consider the graded group
\begin{center}

$\mathcal{M}(E)=\displaystyle \bigoplus_{n\in \ZZ} \mathcal{M}_n(E)$.
\end{center}
This is in fact a module over the ring $\kMW_*(E)$. Now consider $\mathcal{I}(E)=\mathcal{M}(E)\eeta $ the left sub-$\kMW_*(E)$-module generated by $\eeta$. Define $\Gamma^*(M)$ as follows:
\begin{center}

$\Gamma^*(M)(E)=\mathcal{M}(E)/\mathcal{I}(E)$.
\end{center}

This is a $\ZZ$-graded abelian group.
\begin{itemize}
\item (d1) Let $\phi:E\to F$ be a morphism of fields. The collection of maps
\begin{center}

$\mathcal{M}_n(E)\to \mathcal{M}_n(F)$
\end{center}
given by \ref{itm:D1} defines a map 
\begin{center}

$\phi_*:\Gamma^*(M)(E)\to \Gamma^*(M)(F)$
\end{center}
of degree 0.

\item (d2) Let $\phi:E\to F$ be a finite morphism of fields. Denote by $r=\rk \Om_{F/k}=\rk \Om_{E/k}$. For $n\in \ZZ$, We define a map $\phi_\Theta^*$ as in the following commutative diagram.

\begin{center}

$\xymatrix{
\mathcal{M}_n(F) \ar[r]^{\phi_\Theta^*} \ar[d]_\Theta^{\simeq} & \mathcal{M}_n(E) \ar[d]^{\Theta'}_{\simeq} \\
M(F,\Omega_{F/k}+(n-r)\cdot \AAA^1_F) \ar[r]^{\phi^*} & M(E,\Omega_{E/k}+ (n-r)\cdot \AAA^1_E)
}$

\end{center}
where $\Theta$ and $\Theta'$ are isomorphisms induced by $\Om_{F/k}\simeq r\cdot \AAA^1_F$ and $\Om_{E/k}\simeq r\cdot \AAA^1_E$, respectively.
Hence (by \ref{itm:R2a} and \ref{itm:R2b}) we have a morphism of graded groups
\begin{center}

$\phi^*_{\Theta,\Theta'}:\Gamma^*(M)(F)\to \Gamma^*(M)(E)$
\end{center}
of degree 0. By Lemma \ref{LemAdjunction1}, this morphism does not depend on $\Theta, \Theta'$.
	
\item (d3) We have an obvious action of $\kMW_*(E)$ on $\Gamma^*(M)$ where $\eeta$ acts trivially.
\item (d4) Let $E$ be a field with a valuation $v$. Let $\Theta:\NN_v \to \AAA^1_{\kappa(v)}$ be an isomorphism of virtual $\kappa(v)$-vector spaces. For $n\in \ZZ$, this defines a map $\partial_\Theta$ as in the following commutative diagram.
\begin{center}

$\xymatrix{
\mathcal{M}_n(E) \ar[r]^{\partial_\Theta} \ar@{=}[d] & \mathcal{M}_{n-1}(\kappa(v))  \\
M(E,n\cdot \AAA^1_F) \ar[r]^-{\partial} & M(E, -\NN_v+n\cdot \AAA^1_E). \ar[u]^\simeq
}$

\end{center}
Hence (by \ref{itm:R3e}) we have a morphism of graded groups
\begin{center}

$\partial_\Theta:\Gamma^*(M)(E)\to \Gamma^*(M)(\kappa(v))$
\end{center}
of degree $(-1)$. By Lemma \ref{LemAdjunction1}, this morphism does not depend on $\Theta$.

\end{itemize}
\par Since $\epsilon$ modulo $\eeta$ is $(-1)$, we easily check that $\Gamma^*(M)$ satisfies all axioms of Rost's cycle modules, except for \ref{itm:R1c} which can be deduced from Claim \ref{StrongR1c}.
\par Moreover, we can check that this defines a functor
\begin{center}

$\Gamma^*:\mathfrak{M}^{\operatorname{MW}}_k\to \mathfrak{M}^{\operatorname{M}}_k$.
\end{center}

\end{Par}

We now prove the following adjunction theorem (hence our theory extends Rost's).
\begin{The}[Adjunction Theorem]\label{AdjunctionTheorem}

The two previously defined functors form an adjunction ${ \Gamma^*\dashv \Gamma_*}$ between the category of Milnor-Witt cycle modules and the category of classical cycle modules:
\begin{center}
$\Gamma^*:\mathfrak{M}^{\operatorname{MW}}_k \rightleftarrows \mathfrak{M}^{\operatorname{M}}_k:\Gamma_*$.

\end{center}
\end{The}
\begin{proof} Let $M$ be a Milnor-Witt cycle module and $N$ a Rost cycle module. We define a map
\begin{center}
$\Phi:\Hom_{\CatM_k}(\Gamma^*(M),N)\to \Hom_{\CatMW_k}(M,\Gamma_*(N))$.
\end{center}
Let $E$ be a field. By definition, $\Gamma^*(M)(E)$ is a $\ZZ$-graded abelian group; we write:
\begin{center}
${\Gamma^*(M)(E)=\bigoplus_{n\in \ZZ}\Gamma^*(M)(E,n)}$.
\end{center} 
We consider $\Gamma^*(M)$ as a functor $(E,n)\mapsto \Gamma^*(M)(E,n)$ with data $(d1),\dots, (d4)$ and rules $(r1a),\dots, (r3e), (fd)$ and $(c)$. Let $\alpha$ be a morphism from $\Gamma^*(M)$ to $N$. By definition, we have (for any field $E$ and any integer $n$) maps
\begin{center}

$\alpha_{(E,n)}:\Gamma^*(M)(E,n)\to N(E,n)$
\end{center}
compatible (in an obvious sense) with data $(d1),\dots, (d4)$ (see \cite[Definition 1.3]{Rost96}).
\par Let $E$ be a field and $\VV_E$ be a virtual vector bundle of rank $n$ over $E$. We want to define a map:
\begin{center}
$\Phi(\alpha)_{(E,\VV_E)}:M(E,\VV_E)\to \Gamma_*(N)(E,\VV_E)$.
\end{center}
Consider $\Theta:\VV_E\simeq n\cdot \AAA^1_E$ a trivialisation. We have a map $\Phi(\alpha)_E^{\Theta}$ defined as the composite:
\begin{center}
$\xymatrix{
M(E,\VV_E) \ar[r]^{\Theta}\ar[d]^{\Phi(\alpha)_{(E,\VV_E)}^{\Theta}} 
 & M(E,n\cdot \AAA^1_E)\ar[r]^-{\iota_n} 
&
 \bigoplus_{m\in \ZZ} M(E,m\cdot \AAA^1_E) \ar[d]^{\operatorname{mod} \eeta} \\
\Gamma_*(N)(E,\VV_E) &
  \Gamma^*(M)(E,n) \ar[l]^-{\alpha_{(E,n)}}  &
 \Gamma^*(M)(E) \ar[l]^{p_n}  }$
 \end{center}
where $\iota_n$ and $p_n$ are the canonical maps (recall that $\Gamma_*(N)(E,\VV_E)=N(E,n)$ by definition). By Lemma \ref{LemAdjunction1}, the map $\Phi(\alpha)_{(E,\VV_E)}^{\Theta}$ does not depend on $\Theta$ and is denoted by $\Phi(\alpha)_{(E,\VV_E)}$. In order to prove that $\Phi(\alpha):(E,\VV_E)\mapsto \Phi(\alpha)_{(E,\VV_E)}$ is a morphism of Milnor-Witt cycle modules, we have to check that it is compatible with $\ref{itm:D1}$, $\ref{itm:D2}$, $\ref{itm:D3}$ and $\ref{itm:D4}$.
\par We prove compatibility with $\ref{itm:D1}$.
 Let $\phi:E\to F$ be an extension of fields and $\VV_E$ a virtual vector bundle of rank $n$ over $E$. We want to prove that the following diagram (1) is commutative:
% \hspace{-10pt}\xymatrix
% \xymatrix@C=10pt@R=20pt
\begin{center}
$\xymatrix@C=10pt@R=20pt{
M(E,\VV_E) \ar[dd]^<<<<<<{\phi_*} \ar@{}[rddd]|{(1)}
\ar[rr]^>>>>>>>>>>{\Theta}\ar[dr]^{\Phi(\alpha)_{(E,\VV_E)}} &
  &
 M(E,n\cdot \AAA^1_E)\ar[rr]^>>>>>>>>>>{\iota_n} \ar[dd]^<<<<<<{\phi_*}|-{\phantom{f} } &
  &
 \bigoplus_{m\in \ZZ} M(E,m\cdot \AAA^1_E) \ar[dd]|-{\phantom{f}}^<<<<<<{\phi_*} \ar[rd]^{\operatorname{mod} \eeta} & 
\\
  &
\Gamma_*(N)(E,\VV_E) \ar[dd]_<<<<<<{\phi_*} &
  &
\Gamma^*(M)(E,n) \ar[dd]^<<<<<<{\phi_*} \ar[ll]^<<<<<<<<{\alpha_{(E,n)}} &
  & 
\Gamma^*(M)(E) \ar[dd]^<<<<<<{\phi_*} \ar[ll]^<<<<<<<<<<{p_n} 
\\
M(F,\VV_F)  \ar[rr]^>>>>>>>>>>{\Theta}|-{ \, \, \,}
\ar[dr]_{\Phi(\alpha)_{(F,\VV_F)}}&
  &
M(F,n\cdot \AAA^1_F)\ar[rr]|-{\, \, \,}^>>>>>>>>>>{\iota_n}&
  &
 \bigoplus_{m\in \ZZ} M(F,m\cdot \AAA^1_E) \ar[dr]^{\operatorname{mod} \eeta} & 
  \\
  &
\Gamma_*(N)(F,\VV_E) &
  &
\Gamma^*(M)(F,n)  \ar[ll]^<<<<<<<<{\alpha_{(F,n)}}&
  & 
\Gamma^*(M)(F) \ar[ll]^<<<<<<<<<<{p_n} \\}$
\end{center}
where the maps are defined as before. The left back square is commutative according to \ref{itm:R2a} and \ref{itm:R4a}. The left front square is commutative since $\alpha$ is compatible with $(d1)$. The remaining squares are commutative by definition. Hence the square (1) is commutative.
\par We prove compatibility with $\ref{itm:D2}$.
 Let $\psi:F\to E$ be a finite extension of fields and $\VV_F$ a virtual vector bundle of rank $n$ over $F$. Denote by $r$ the rank of $\Om_{E/k}$. We want to prove that the following diagram (2) is commutative:
\begin{center}
%\resizebox{18cm}{!}{}
$\xymatrix@C=0pt@R=20pt{
M(E,\Om_{E/k}+\VV_E) \ar[dd]^<<<<<<{\psi^*} \ar@{}[rddd]|{(2)}
\ar[rr]^>>>>>>>>>>{\Theta}\ar[dr]^{\Phi(\alpha)_{(E,\Om_{E/k}+\VV_E)}} &
  &
 M(E,(r+n)\cdot \AAA^1_E)\ar[rr]^>>>>>>>>>>{\iota_{(r+n)}} \ar[dd]^<<<<<<{\psi^*}|-{\phantom{f} } &
  &
 \bigoplus_{m\in \ZZ} M(E,m\cdot \AAA^1_E) \ar[dd]|-{\phantom{f}}^<<<<<<{\psi^*} \ar[rd]^{\operatorname{mod} \eeta} & 
\\
  &
\Gamma_*(N)(E,\Om_{E/k}+\VV_E) \ar[dd]_<<<<<<{\psi^*} &
  &
\Gamma^*(M)(E,{(r+n)}) \ar[dd]^<<<<<<{\psi^*} \ar[ll]^<<<<<<<<{\alpha_{(E,r+n)}} &
  & 
\Gamma^*(M)(E) \ar[dd]^<<<<<<{\psi^*} \ar[ll]^<<<<<<<<<<{p_{(r+n)}} 
\\
M(F,\Om_{F/k}+\VV_F)  \ar[rr]^>>>>>>>>>>{\Theta}|-{ \, \, \,}
\ar[dr]_{\Phi(\alpha)_{(F,\Om_{F/k}+\VV_F)}}&
  &
M(F,(r+n)\cdot \AAA^1_F)\ar[rr]|-{\, \, \,}^>>>>>>>>>>{\iota_{(r+n)}}&
  &
 \bigoplus_{m\in \ZZ} M(F,m\cdot \AAA^1_E) \ar[dr]^{\operatorname{mod} \eeta} & 
  \\
  &
\Gamma_*(N)(F,\Om_{E/k}+\VV_E) &
  &
\Gamma^*(M)(F,{(r+n)})  \ar[ll]^<<<<<<<<{\alpha_{(F,r+n)}}&
  & 
\Gamma^*(M)(F) \ar[ll]^<<<<<<<<<<{p_{(r+n)}} \\}$

\end{center}
where the maps are defined as before. The left back square is commutative according to \ref{itm:R2b} and \ref{itm:R4a}. The left front square is commutative since $\alpha$ is compatible with $(d2)$. The remaining squares are commutative by definition. Hence the square (2) is commutative.

\par We prove compatibility with \ref{itm:D3}. Let $E$ be a field and $\VV_E$ a virtual vector space of dimension $n$ over $E$. Let $u\in E^{\times}$ be a unit and consider multiplication with $[u]\in \KMW(E,\AAA^1)$. We want to prove that the following diagram (3) is commutative:
\begin{center}
% \resizebox{18cm}{!}{}
$\xymatrix@C=0pt@R=20pt{
M(E,\VV_E) \ar[dd]^<<<<<<{\gamma_{[u]}} \ar@{}[rddd]|{(3)}
\ar[rr]^>>>>>>>>>>{\Theta}\ar[dr]^{\Phi(\alpha)_{(E,\VV_E)}} &
  &
 M(E,n\cdot \AAA^1_E)\ar[rr]^>>>>>>>>>>{\iota_n} \ar[dd]^<<<<<<{\gamma_{[u]}}|-{\phantom{f} } &
  &
 \bigoplus_{m\in \ZZ} M(E,m\cdot \AAA^1_E) \ar[dd]|-{\phantom{f}}^<<<<<<{\gamma_{[u]}} \ar[rd]^{\operatorname{mod} \eeta} & 
\\
  &
\Gamma_*(N)(E,\VV_E) \ar[dd]_<<<<<<{\gamma_{[u]}} &
  &
\Gamma^*(M)(E,n) \ar[dd]^<<<<<<{\gamma_{[u]}} \ar[ll]^<<<<<<<<{\alpha_{(E,n)}} &
  & 
\Gamma^*(M)(E) \ar[dd]^<<<<<<{\gamma_{[u]}} \ar[ll]^<<<<<<<<<<{p_n} 
\\
M(E,\AAA^1_E+\VV_E)  \ar[rr]^>>>>>>>>>>{\Theta}|-{ \, \, \,}
\ar[dr]_{\Phi(\alpha)_{(E,\AAA^1_E+\VV_E)}}&
  &
M(E,(n+1)\cdot \AAA^1_E)\ar[rr]|-{\, \, \,}^>>>>>>>>>>{\iota_{n+1}}&
  &
 \bigoplus_{m\in \ZZ} M(E,m\cdot \AAA^1_E) \ar[dr]^{\operatorname{mod} \eeta} & 
  \\
  &
\Gamma_*(N)(E,\AAA^1_E+\VV_E) &
  &
\Gamma^*(M)(E,n+1)  \ar[ll]^<<<<<<<<{\alpha_{(E,n+1)}}&
  & 
\Gamma^*(M)(E) \ar[ll]^<<<<<<<<<<{p_{n+1}} \\}$

\end{center}
where the maps are defined as before. The left back square is commutative according to \ref{itm:R4a}. The left front square is commutative since $\alpha$ is compatible with $(d3)$. The remaining squares are commutative by definition. Hence the square (3) is commutative. The same proof works for the generator $\eeta$, hence $\Phi(\alpha)$ is compatible with \ref{itm:D3}.

\par We prove compatibility with \ref{itm:D4}. Let $E$ be a field over $k$, let $v$ be a valuation on $E$ and let $\mathcal{V}$ be a virtual projective {$\mathcal{O}_v$-module} of finite type. Denote by $\mathcal{V}_E=\VV \otimes_{\mathcal{O}_v} E$ and $\VV_{\kappa(v)}=\VV \otimes_{\mathcal{O}_v} \kappa(v)$. We want to prove that the following diagram (4) is commutative:
\begin{center}
\resizebox{18cm}{!}{
$\xymatrix@C=0pt@R=20pt{
M(E,\VV_E) \ar[dd]^<<<<<<{\partial_v} \ar@{}[rddd]|{(4)}
\ar[rr]^>>>>>>>>>>{\Theta}\ar[dr]^{\Phi(\alpha)_{(E,\VV_E)}} &
  &
 M(E,n\cdot \AAA^1_E)\ar[rr]^>>>>>>>>>>{\iota_n} \ar[dd]^<<<<<<{\partial_v}|-{\phantom{f} } &
  &
 \bigoplus_{m\in \ZZ} M(E,m\cdot \AAA^1_E) \ar[dd]|-{\phantom{f}}^<<<<<<{\partial_v} \ar[rd]^{\operatorname{mod} \eeta} & 
\\
  &
\Gamma_*(N)(E,\VV_E) \ar[dd]_<<<<<<{\partial_v} &
  &
\Gamma^*(M)(E,n) \ar[dd]^<<<<<<{\partial_v} \ar[ll]^<<<<<<<<{\alpha_{(E,n)}} &
  & 
\Gamma^*(M)(E) \ar[dd]^<<<<<<{\partial_v} \ar[ll]^<<<<<<<<<<{p_n} 
\\
M(\kappa(v),-\NN_v+\VV_{ \kappa(v)})  \ar[rr]^>>>>>>>>>>{\Theta}|-{ \, \, \,}
\ar[dr]_{\Phi(\alpha)_{(F,\VV_F)}}&
  &
M(\kappa(v),(n-1)\cdot \AAA^1_{\kappa(v)})\ar[rr]|-{\, \, \,}^>>>>>>>>>>{\iota_{n-1}}&
  &
 \bigoplus_{m\in \ZZ} M(\kappa(v),m\cdot \AAA^1_{\kappa(v)}) \ar[dr]^{\operatorname{mod} \eeta} & 
  \\
  &
\Gamma_*(N)(\kappa(v),-\NN_v+\VV_{ \kappa(v)}) &
  &
\Gamma^*(M)(\kappa(v),n-1)  \ar[ll]^<<<<<<<<{\alpha_{(\kappa(v),n-1)}}&
  & 
\Gamma^*(M)(\kappa(v)) \ar[ll]^<<<<<<<<<<{p_{n-1}} \\}$
 }

\end{center}
where the maps are defined as before. The left back square is commutative according to \ref{itm:R4a}. The left front square is commutative since $\alpha$ is compatible with $(d4)$. The remaining squares are commutative by definition. Hence the square (4) is commutative.

\begin{comment}
\par We can prove in a similar way compatibility with $\ref{itm:D2}$, $\ref{itm:D3}$ and $\ref{itm:D4}$.
\end{comment}
\par Thus, we have defined a map
\begin{center}
$\Phi:\Hom_{\CatM_k}(\Gamma^*(M),N)\to \Hom_{\CatMW_k}(M,\Gamma_*(N))$ \\
$\alpha\mapsto \Phi(\alpha)$
\end{center}
which is natural in $M$ and $N$.
\par  Let $M$ be a Milnor-Witt cycle module and $N$ a Rost cycle module, we now define a map
\begin{center}
$\Psi:\Hom_{\CatMW_k}(M,\Gamma_*(N)) \to\Hom_{\CatM_k}(\Gamma^*(M),N) $.
\end{center}
Let $\beta$ be an element of $\Hom_{\CatMW_k}(M,\Gamma_*(N))$. Let $E$ be a field. By definition, we have maps\footnote{The fact that $\beta$ commutes with \ref{itm:D1} implies that these maps depend only on the rank $n$ of $\VV_E$.}\begin{center}
$\beta_{(E,n)}:M(E,\VV_E) \to N(E,n)$
\end{center}
for any virtual vector space $\VV_E$ of rank $n$ over $E$. These maps are compatible with data $\ref{itm:D1}$, $\ref{itm:D2}$, $\ref{itm:D3}$ and $\ref{itm:D4}$ according to Definition \ref{DefMWmorphisms}.
 Thus, we can define a map
\begin{center}
$\bigoplus_{n\in \ZZ} \beta_{(E,n)}:
\bigoplus_{n\in \ZZ} M(E,n\cdot \AAA^1_E)\to 
\bigoplus_{n\in \ZZ} N(E,n)$\end{center}
which can be factorized as
\begin{center}
$\Psi(\beta)_E:\Gamma^*(M(E)) \to N(E)=\bigoplus_{n\in \ZZ} N(E,n)$
\end{center}
since $\beta$ commutes with the $\KMW$-action and $\eeta$ acts trivially on the Milnor-Witt cycle module $\Gamma_*(N)$ by definition. Since $\beta$ commutes with data $\ref{itm:D1}$, $\ref{itm:D2}$, $\ref{itm:D3}$ and $\ref{itm:D4}$, we can prove that the map ${\Psi(\beta):E\mapsto \Psi(\beta)_E}$ is a morphism of Rost cycle modules.
\par Thus, we have defined a map
\begin{center}
$\Psi:\Hom_{\CatMW_k}(M,\Gamma_*(N)) \to\Hom_{\CatM_k}(\Gamma^*(M),N) $
\\ $\beta \mapsto \Psi(\beta)$
\end{center}
which is natural in $M$ and $N$.
\par Since the map $\Psi$ is inverse to $\Phi$, we conclude that the two functors $\Gamma^*$ and $\Gamma_*$ form an adjunction ${ \Gamma^*\dashv \Gamma_*}$ between the category of Milnor-Witt cycle modules and the category of classical cycle modules:
\begin{center}
$\Gamma^*:\mathfrak{M}^{\operatorname{MW}}_k \rightleftarrows \mathfrak{M}^{\operatorname{M}}_k:\Gamma_*$.

\end{center}

\end{proof}

\appendix
\section{Virtual Objects}\label{VirtualObj}

In this section, we closely follow \cite[§4]{Deligne87} and show how to construct the category $\mathfrak{V}(\mathcal{A})$ of virtual objects from an exact category $\mathcal{A}$.
\begin{Par}

Recall that an exact category is an additive category $\mathcal{A}$ equipped with a set of exact sequences $X\rightarrowtail Y \twoheadrightarrow Z$ satisfying various axioms (see \cite[Definition 2.1]{Buhler10}). Our main example consists in the category of vector bundles over a scheme $X$ (with locally split short exact sequences).
\end{Par}
\begin{Par}
A (commutative) Picard category is a non-empty category $\mathcal{P}$ where all arrows are isomorphisms, with a functor $+:\mathcal{P}\times \mathcal{P}\to \mathcal{P}$ following some associativity and commutativity constraints, and such that for any object $P$, the two functors $X\mapsto X+P$ and $X\mapsto P+X$ are autoequivalences of $\mathcal{P}$.
\par One can show that there is a zero object $0$ (unique up to unique isomorphism) and that any object $X$ has an inverse $-X$ (unique up to unique isomorphism) such that $X+(-X)\simeq 0$.
The commutativity condition gives switch isomorphisms $X+Y\simeq Y+X$ compatible with the associative data. Beware that these switch isomorphisms are not identities in general.
\par The functor $X\mapsto -X$ is compatible with $+$ as follows: the isomorphism $(X+Y)+((-X)+(-Y))\simeq (X+(-X))+(Y+(-Y))\simeq 0+0\simeq 0$ is such that $(-X)+(-Y)$ is (isomorphic to) $-(X+Y)$. Nevertheless, beware that the diagram
\begin{center}

$\xymatrix{
((-X)+X)+(-X) \ar[rr]^\simeq \ar[d]^\simeq&   &(-X)+(X+(-X)) \ar[d]^\simeq \\
0+(-X) \ar[r]^\simeq & (-X) \ar[r]^\simeq & (-X)+0
}$
\end{center}
is not commutative.
\end{Par}
\begin{Exe} [Graded line bundles, see also \cite{Fasel18bis}, §1.3]
\label{ExGradedLines}
Let $S$ be a scheme and $\mathcal{L}(S,\ZZ)$ be the category of graded line bundles over $S$. An object of this category is a pair $(a,L)$ where $L$ is a line bundle (an invertible $\mathcal{O}_X$-module) and $a$ is a locally constant integer (an element of $\Gamma(S,\ZZ)$). A map $(a,L)\to (a',L')$ between two such objects is an isomorphism $L\to L'$ of line bundles with the data $a=a'$. By definition, we have $(a,L)+(a',L')=(a+a',L\otimes L')$. The associativity constraint is deduced from the one ruling the tensor product and the commutativity constraint $(a,L)+(a',L')\simeq (a',L')+(a,L)$ is given by $l\wedge l' \mapsto (-1)^{aa'}l'\wedge l$ (also called {\em Koszul sign convention}). The unit object $0$ is $(0,\mathcal{O}_X)$. For any object $(a,L)$, there is an isomorphism $L\otimes L^{\vee}\simeq \mathcal{O}_X$ induced by the pairing $L\times L^{\vee}\to \mathcal{O}_X, (l,\phi)\mapsto \phi(l)$ (where $L^{\vee}$ denotes the dual of $L$). Hence the following isomorphism:
\begin{center}

$(a,L)+(-a,L^{\vee})\simeq (0,\mathcal{O}_X)$.
\end{center}
\end{Exe}
\begin{Par} Let $\mathcal{A}$ be an exact category. We construct the category $\mathfrak{V}(\mathcal{A})$ of virtual objects of $A$. Denote by $\mathcal{A}_{iso}$ the category with the same objects as $\mathcal{A}$ but with morphisms the isomorphisms of $\mathcal{A}$. Let $\mathcal{P}$ be a Picard category. Consider the functors $[-]:\mathcal{A}_{iso}\to \mathcal{P}$ with data \ref{itm:(a)}, \ref{itm:(b)} and rules \ref{itm:(c)}, \ref{itm:(d)}, \ref{itm:(e)} described as follows:

\begin{description}
\item[\namedlabel{itm:(a)}{(a)}](additivity) For any short exact sequence $\Sigma:A'\rightarrowtail A \twoheadrightarrow A''$, we have a morphism $[\Sigma]:[A]\to [A']+[A'']$, functorial in morphisms of short exact sequences.
\item[\namedlabel{itm:(b)}{(b)}] If $0$ is a zero object in $\mathcal{A}$, then we have an isomorphism $[0]\to 0$.
\item[\namedlabel{itm:(c)}{(c)}] Let $\phi:A\to B$ be an isomorphism in $\mathcal{A}$. Denote by $\Sigma$ the short exact sequence $0\to A\to B$ (resp. $A\to B \to 0$). Then $[\phi]$ (resp. $[\phi]^{-1}$) is the composite
\begin{center}

$\xymatrix{
[A] \ar[r]^-{[\Sigma]} & [0]+[B] \ar[r]^-{(b)} & [B]
}$
\end{center}
(resp.
\begin{center}

$\xymatrix{
[B] \ar[r]^-{[\Sigma]} & [B]+[0] \ar[r]^-{(b)} & [A]).
}$
\end{center}

\item[\namedlabel{itm:(d)}{(d)}](associativity) Consider the following filtration $C\supset B \supset A \supset 0$. Then the diagram (with morphisms given by \ref{itm:(a)})
\begin{center}
$\xymatrix{
[C] \ar[r] \ar[d] & [A]+[C/A] \ar[d] 
\\
[B]+[C/B] \ar[r] &[A]+[B/A]+[C/B]
}$
\end{center}
is commutative.

\item[\namedlabel{itm:(e)}{(e)}](\ref{itm:(a)} is additive) For any $A=A'\oplus A''$, the short exact sequences $\Sigma:A'\to A \to A''$ and $\Sigma':A''\to A \to A'$ give a commutative triangle
\begin{center}

$\xymatrix{
[A']+[A''] \ar[rr] \ar[rd]_{[\Sigma]} & & [A'']+[A'] \ar[ld]^{[\Sigma']} \\
 & [A].
 }$
\end{center}
\end{description}
For any commutative Picard category $\mathcal{P}$, the functors 
[$-]:\mathcal{A}\to \mathcal{P}$ satisfying the conditions \ref{itm:(a)},...,\ref{itm:(e)} define a category denoted by $[[-]]_\mathcal{P}$. One can prove that there is a (commutative) Picard category $\mathfrak{V}(A)$ with functors $[-]_{\mathcal{A}}:\mathcal{A}_{iso}\to \mathfrak{V}(A)$ satisfying the conditions \ref{itm:(a)},...,\ref{itm:(e)} which is universal in the sense that for any Picard category $\mathcal{P}$ with functors $[-]:\mathcal{A}_{iso}\to \mathcal{P}$ satisfying \ref{itm:(a)},...,\ref{itm:(e)}, the category $[[-]]_\mathcal{P}$ is equivalent to the category of additive functors $\mathfrak{V}(A)\to \mathcal{P}$. By definition, $\mathfrak{V}(A)$ is the category of virtual objects of $\mathcal{A}$.

\end{Par}
\begin{Par}

Let $T:\mathcal{A}\to \mathcal{B}$ be an exact functor between two exact categories (it respects short exact sequences). With the previous notations, the composition $[-]\circ T:\mathcal{A}_{iso}\to \mathfrak{V}(\mathcal{B})$ is a functor satisfying \ref{itm:(a)},...,\ref{itm:(e)} thus it induces an additive functor $\mathfrak{V}(\mathcal{A})\to \mathfrak{V}(\mathcal{B})$.
\end{Par}
\begin{Par}

Let $S$ be a scheme and denote by $\operatorname{Vect}(S)$ the category of vector bundle over $S$. We put $\mathfrak{V}(S)=\mathfrak{V}(\operatorname{Vect}(S))$ (also denoted by $\underline{K}(S)$ in \cite[§4]{Deligne87}). Any scheme morphism $f:X\to S$ defines a pullback $f^*
:\mathfrak{V}(S)\to \mathfrak{V}(X)$. 
\end{Par}
\begin{Par}
Let $S$ be a scheme and $V$ be a vector bundle over $S$. The rank $\rk(V)$ of $V$ is a locally constant integer on $S$. The determinant $\det(V)$ of $V$ is the line bundle $\Lambda^{\rk(V)}(V)$. We also call determinant of $V$ (and write $\Det(V)$) the graded line bundle $(\rk(V),\det(V))$. For any short exact sequence
\begin{center}

$V' \rightarrowtail V \twoheadrightarrow V''$,
\end{center}
we have an isomorphism
\begin{center}

$\det(V')\otimes \det(V'')\to \det(V)$
\end{center}
illustrated by the symbols
\begin{center}

$(e'_1\wedge\dots \wedge e'_n)\otimes (e''_1\wedge \dots \wedge e''_m) \mapsto e'_1\wedge\dots \wedge e'_n \wedge \tilde{e}_1''\wedge \dots \wedge \tilde{e}''_m$.
\end{center}
This induces an isomorphism between the associated line bundles which satisfies conditions \ref{itm:(a)},...,\ref{itm:(e)}. The universal property gives a factorization of $\Det:\operatorname{Vect}(S)\to \mathcal{L}(S,\ZZ)$ through a functor $\mathfrak{V}(S)\to \mathcal{L}(S,\ZZ)$ also denoted by $\Det$. 

\par Note in particular that for a vector bundle $V=V'\oplus V''$, the two short exact sequences give the following commutative diagram:
\begin{center}

$\xymatrix{
\det(V')\otimes \det(V'') \ar[rr]^{(-1)^{nm}} \ar[rd]^\simeq & & \det(V'')\otimes \det(V') \ar[ld]_\simeq \\
 & \det(V) & 
 }$
\end{center}
where $n,m$ are the rank of $V',V''$, respectively.
\par 
If $\mathcal{V}$ is a virtual vector bundle over $S$ (i.e. an object of $\mathfrak{V}(S)$), then we can write the graded line bundle $\Det(\mathcal{V})$ as
 $(\rk(\mathcal{V}),\det(\mathcal{V}))$ so that we can define the rank and the determinant of $\mathcal{V}$ in an obvious way.
 \par In this article, we are mostly interested in the case where $S=\Spec(A)$ is a local scheme. One can then prove that the functor $\Det:\mathfrak{V}(S)\to \mathcal{L}(S,\ZZ)$ is an equivalence. 
\end{Par}

%          
%          \section{Appendix}
%          \subsection{Kähler Differentials Modules}
%          
%          \begin{Pro} \label{KahlerDiffFractionField}
%          
%          Let $F$ be a field over $k$ and denote $F(u)$ its function field. Then we have a canonical isormophism
%          \begin{center}
%          
%          $\Om_{F(u)/k}\simeq \Om_{F/k}\otimes_F F(u) + F(u)du$
%          \end{center}%          
%          as $F(u)$-vector spaces.
%          \end{Pro}
%          
%          \begin{Pro} \label{KahlerDiffImmersion}
%          
%          If $Z\to X$ is an immersion of schemes, then
%          \begin{center}
%          
%          $\TT_{Z/X}=0$.
%          \end{center}
%          \end{Pro}
%          
%          \begin{Pro}
%          
%          Let $Z\to X$ be an immersion of schemes. Then $\TT_{Z/X}=0$.
%          \end{Pro}
%          

  \bibliographystyle{alpha}
  \bibliography{BiblioFile3}

%\bibliographystyle{plain} % D'autres styles sont disponibles. Notez que les distributions LaTeX n'incluent généralement pas de styles de bibliographies francisés ; vous aurez donc des bibliographies en anglais.
%\bibliography{biblioo} % Remplacer "biblio" par le nom de votre fichier de références bibliographiques.

\end{document}